\documentclass[11pt,reqno]{amsart}
\usepackage{amsmath,amsthm,amsfonts,amssymb,amscd,overpic}
\numberwithin{equation}{section}

\newcommand{\eqdef}{\stackrel{\scriptscriptstyle\rm def}{=}}
\def\ZZ {{\mathbb Z}}\def\La{\Lambda}\def\Si{\Sigma}
\def\NN {{\mathbb N}}\def\CC {{\mathbf C}}\def\RR {{\mathbb R}}
\def\Ga{\Gamma}\def\be{\beta}\def\la{\lambda}\def\ga{\gamma}\def\de{\delta}
\def\De{\Delta}\def\Ups{\Upsilon}\def\bQ{Q^\ast}
\def\cL{{\mathcal I}}\def\cA{{\mathcal A}}\def\cM{{\mathcal M}}
\def\bR{\mathbb R}\def\sst{{ss}}\def\uut{{uu}}\def\st{s}\def\ut{u}\def\loc{{\operatorname{loc}}}\def\cut{{cu}}

\newtheorem{theo}{Theorem}

\newtheorem{theor}{Theorem}[section]
\newtheorem{theoo}[theor]{Theorem}
\newtheorem{lemm}[theor]{Lemma}
\newtheorem{claim}[theor]{Claim}
\newtheorem{coro}[theor]{Corollary}

\newtheorem{defi}[theor]{Definition}
\newtheorem{prop}[theor]{Proposition}

\newtheorem{rema}[theor]{Remark}

\newtheorem{ques}[theor]{Question}

\newtheorem{nota}{Notation}[section]

\newlength\Li \newlength\Lii
\title[Porcupine-like horseshoes]{Porcupine-like horseshoes:\\ Transitivity, Lyapunov spectrum,\\ and phase transitions}
\author{Lorenzo J. D\'\i az}
\address{Departamento de Matem\'atica PUC-Rio, Marqu\^es de S\~ao Vicente 225, G\'avea, Rio de Janeiro 225453-900, Brazil}
\email{lodiaz@mat.puc-rio.br}
\author{Katrin~Gelfert}
 \address{Instituto de Matem\'atica UFRJ, Av. Athos da Silveira Ramos 149, Cidade Universit\'aria - Ilha do Fund\~ao, Rio de Janeiro 21945-909,  Brazil}
\email{gelfert@im.ufrj.br}

\thanks{This paper was partially supported by CNPq, Faperj, Pronex (Brazil) and the Alexander von Humboldt Foundation (Germany). The authors thank the Departamento de Matem\'atica Pura, Universidade de Porto, and IM\,PAN Warsaw for support and hospitality, as well as the participants of the DynNonHyp ANR-meeting Dijon, June 2010, M.~Rams for useful comments, and an  anonymous referee for careful reading and relevant suggestions.}
\keywords{homoclinic class, Lyapunov exponent, non-contracting iterated function system, partial hyperbolicity, phase transition, spectral gap, skew-product, transitive set}
\subjclass[2000]{%Primary:%
37D35, % Thermodynamic formalism, variational principles, equilibrium states
37D25, %Nonuniformly hyperbolic systems (Lyapunov exponents, Pesin theory, etc.)
37E05, % Maps of the interval (piecewise continuous, continuous, smooth)
37D30, % partially hyperbolic systems and dominated splittings
37C29 %Homoclinic and heteroclinic orbits
}

\begin{document}

\maketitle
\begin{abstract}
	We study a partially hyperbolic and topologically transitive local diffeomorphism $F$ that is a skew-product over a horseshoe map. This system is derived from a homoclinic class and contains infinitely many hyperbolic periodic points of different indices and hence is not hyperbolic. The associated transitive invariant set $\Lambda$ possesses a very rich fiber structure, it contains uncountably many trivial and uncountably many non-trivial fibers. Moreover, the spectrum of the central Lyapunov exponents of $F|_{\Lambda}$ contains a gap and hence gives rise to a first order phase transition.
	A major part of the proofs relies on the analysis of an associated iterated function system that is genuinely non-contracting.
\end{abstract}

\tableofcontents

%-------------------------------------------------------------------------------------------------------
\section{Introduction}
%-------------------------------------------------------------------------------------------------------

In this paper we  provide examples of non-hyperbolic transitive sets, that we call \emph{porcupine-like horseshoes} or, briefly, \emph{porcupines}, that show a rich  dynamics although they admit a quite simple formulation.
Their dynamics are conjugate to skew-products over a shift whose fiber dynamics  are given by genuinely non-contracting iterated function systems (IFS) on the unit interval.

Naively, from a topological point of view, a porcupine is a transitive set that looks like a horseshoe with infinitely many spines attached at various levels and in a dense way. In terms of its hyperbolic-like structure, it is a partially hyperbolic set with a one-dimensional center, whose spectrum of central Lyapunov exponents contains an interval with negative and positive values which, in particular, illustrates that the porcupine is non-hyperbolic. Although the dynamics on the porcupine is transitive, its spectrum of central exponents has a gap and thus gives rise to a first order phase transition.

Our goal is to present these examples  and to explore their dynamical properties. We are not aiming for the most general setting possible, but instead want to present  the ideas behind our constructions. We think that these examples are representative providing models for a number of key properties of non-hyperbolic dynamics.

%-------------------------------------------------------------------------------------------------------
\subsection{Non-contracting iterated function systems}\label{subsec:1}
%-------------------------------------------------------------------------------------------------------

The analysis in this paper is essentially built on properties of a certain class of non-contracting iterated function systems associated to the central dynamics of the porcupine.

We consider $f_0$, $f_1\colon[0,1]\to\bR$ that are $C^k$ smooth, $k\ge1$, and satisfy
\begin{enumerate}	
	\item[] $f_0$ is orientation preserving, has an expanding fixed point $q$, a contracting fixed point $p>q$, and no further fixed points in $(q,p)$,
	\item[] $f_1$ is an orientation reversing contraction
\end{enumerate}
that form an open set in the corresponding product topology. And we study the co-dimension 1 sub-manifold of maps satisfying the \emph{($q,p$)-cycle condition}
\[
	f_1(p)=q
\]
(compare Figure~\ref{fig.ifs}).

Now we consider compositions of maps $f_i$, $i=0$, $1$. Given a sequence $\xi=(\ldots\xi_{-1}.\xi_0\xi_1\ldots)\in\Sigma_2\eqdef\{0,1\}^\ZZ$, for a point $x\in [q,p]$, we define the \emph{(forward) Lyapunov exponent} of the IFS generated by $f_0$, $f_1$ at $(x,\xi)$ by
\[
	\chi(x,\xi)\eqdef
	\lim_{n\to\infty}\frac 1 n \log
	\big\lvert\big(f_{[\xi_0\ldots\xi_n]} \big)'(x)\big\rvert, \quad
	\text{ where }
	f_{[\xi_0\ldots\xi_n]}\eqdef f_{\xi_n}\circ\cdots\circ f_{\xi_0}.
\]
We restrict our considerations to points that remain in the interval $[q,p]$ under forward and backward iterations. For that we define the \emph{admissible domain} $I_\xi\subset[q,p]$ by
\[
	I_\xi \eqdef
	\bigcap_{m\ge1}
	(f_{\xi_{-1}}\circ\cdots\circ f_{\xi_{-m}})\big([q,p]\big).
\]

 We obtain the following auxiliary result, that is an one-dimensional version of the main result in this paper  (Theorem~\ref{t.bonyexample}).

\begin{theo}\label{theorem0}
	For the IFS generated by maps $f_0$, $f_1$ as above satisfying the ($q,p$)-cycle condition, we have:
	\begin{enumerate}
	\item [(A)]  There are is an uncountable and dense set of sequences $\xi\in\Sigma_2$ so that the admissible domain $I_\xi$ is non-trivial. There is a residual set of sequences $\xi\in\Sigma_2$ so that $I_\xi$ contains a single point only.
	\medskip
	\item [(B)] The points that are fixed with respect to $f_{[\xi_0\ldots\xi_m]}$ for certain $\xi\in\Sigma_2$ and $m\ge0$ are dense in $[q,p]$. Moreover, there exist $x\in(q,p)$ and $\xi\in\Sigma_2$ such that $\{f_{[\xi_0\ldots\xi_m]}(x)\}_{m\ge0}$ is dense in $[q,p]$.
	\medskip
	\item [(C)] There exists $\rho\in(0,\log\,\lvert f'_0(q)\rvert)$ so that the spectrum of all possible Lyapunov exponents is contained in
	\[
		\big[\log\,\lvert f'_0(p)\rvert,\rho\big]\cup\big\{\log\,\lvert f'_0(q)\rvert\big\}.
	\]
\end{enumerate}	
\end{theo}

The cycle condition seems to play a role similar to the Misiurewicz property in one-dimensional dynamics best illustrated by the behavior of the quadratic map $f(x)= 1-2x^2$ that has some alike features~\footnote{A differentiable interval map satisfies the \emph{Misiurewicz condition} if the forward orbit of a critical point does not accumulate onto critical points. Note that $f$ is conjugate to the tent map in $[-1,1]$
and that this conjugation is differentiable in $(-1,1)$. Thus, in particular, the spectrum of the Lyapunov exponents of $f$ contains only the two values $2\log 2$ and $\log 2$}. However, we point out that in our case the breaking of hyperbolicity and the spectral gap are not caused by any critical behavior. Note also that in our case the spectrum is richer and contains a continuum with positive and negative values.
In our case, the cycle condition and the fact that $f_1$ is orientation reversing allows transitivity. However, typical orbits only slowly approach the cycle points $p$ (which corresponds to the critical point) and $q$ (which corresponds to the post-critical point) giving rise to some  transient behavior and hence to the gap in the spectrum.

It would be interesting to find other representative examples that show a gap in the Lyapunov spectrum and hence indicating the presence of a first order phase transitions (the associated pressure function is not differentiable, see Proposition~\ref{p.phasetrans}). We believe that this point deserves special attention. A collection of examples that present phase transitions are provided in~\cite{IomTod:} in the case of interval maps and~\cite{Oli:99} for abstract shift spaces.

%-------------------------------------------------------------------------------------------------------
\subsection{Non-hyperbolic transitive homoclinic classes}
%-------------------------------------------------------------------------------------------------------

We now put the above abstract results into the framework of local diffeomorphisms and return to the analysis of porcupines.

The porcupines considered in this paper are in fact homoclinic classes (see Definitions~\ref{def:porc} and~\ref{def:homcla}) that contain infinitely many saddles of different \emph{indices} (dimension of the unstable direction) scattered throughout the class preventing hyperbolicity. Moreover, they exhibit a rich topological structure in their fibers (which are tangent to the central direction): there are uncountably many fibers whose intersections with the porcupine are continua and infinitely many fibers whose intersections with the porcupine are just points.
Further,  the spectrum of the central Lyapunov exponents  of these sets has a gap and contains an interval (containing positive and negative values). These properties will be stated in Theorem~\ref{t.bonyexample} that is a higher-dimensional version of Theorem~\ref{theorem0} for local diffeomorphisms.

We point out that porcupines also have strong indications to show  a lot of genuinely non-hyperbolic properties that will be explored elsewhere. For example, the transitive porcupines that we construct  do not possess the shadowing property and, following constructions in~\cite{GorIlyKleNal:05,DiaGor:09,BonDiaGor:10}, one can show that they carry non-hyperbolic ergodic measures with large supports and, in view of~\cite{BonDia:08}, we expect that they also display robust heterodimensional cycles.

Let us point out some further motivation.
Il'yashenko in his lecture~\cite{Ily:10} presented topological examples of fibered systems over a shift map that possess recurrent sets (that he calls {\emph{bony sets}}) containing some fibers.
The transitive porcupines provide examples of smooth realizations of such systems. Kudryashov~\cite{Kud:10,Kud:10b} obtained recently a quite general open class of smooth skew-product systems which exhibit bony attractors.
Our examples are also motivated by the construction in~\cite{DiaHorRioSam:09} of bifurcating homoclinic classes and the subsequent study of their Lyapunov spectrum in~\cite{LepOliRio:}. These sets show indeed porcupine-like features but are ``essentially hyperbolic'' in the sense that all their ergodic measures are hyperbolic.

Let us now point out two topological properties of our
examples (in fact also present in~\cite{DiaHorRioSam:09}):
\begin{enumerate}
\item[1)] the porcupine is the homoclinic class $H( Q^\ast,F)$ of a
saddle $ Q^\ast$ that contains two fixed points $P$ and $Q$ of
different indices that are related by a heterodimensional cycle
(see Definition~\ref{def:hetcyc}),
\item[2)] the saddle $Q$ has the same index as $ Q^\ast$, but is not homoclinically related to $Q$ (compare Definition~\ref{def:homcla}).
\end{enumerate}
Comparing with the one-dimensional setting in Section~\ref{subsec:1} the points $P$ and $Q$ play the role of $p$ and $q$, and the heterodimensional cycle corresponds to the cycle property.
Our examples are  ``essentially non-hyperbolic'': The porcupines contain infinitely many saddles of different indices.
This property (and the fact that the porcupine is transitive) is the main reason for the fact that the central Lyapunov spectrum contains an interval with positive and negative values. We remark that 2) is also the main reason for the presence of a gap in the spectrum of the central Lyapunov exponents in~\cite{LepOliRio:}. In fact, the condition about the saddle
$Q$ in 2) is a \emph{necessary} condition to obtain such a gap
(see Lemma~\ref{l.homrelspec}). In our example the
spectrum of the Lyapunov exponent associated to the central
direction contains a continuum with positive and negative values
and a separated point. This is as far as we know the first example
with a spectral gap that is essentially non-hyperbolic and is not related to the occurrence of critical points.

 We are aware of the fact that sets that display properties 1) and 2) above are quite specific: By the Kupka-Smale theorem saddles of generic diffeomorphisms are not related by heterodimensional cycles and after~\cite{AbdBonCroDiaWen:07}  property 2) is $C^1$ non-generic. But this clearly does not imply that the classes discussed here are not representative.

Finally, we would like to mention that this paper proceeds a systematic study of non-hyperbolic homoclinic classes. Besides the above references, we would like to mention~\cite{AbdBonCro:}, \cite{DiaGor:09}, and \cite{BonDiaGor:10} where ergodic properties (related to the Lyapunov spectrum) of homoclinic classes are stated.

Before presenting our results, let us state precisely the main objects we are going to study.

\begin{defi}[Partial hyperbolicity]\label{def:parthyp}{\rm
    An $F$-invariant compact set $\Lambda$ is said to be \emph{partially hyperbolic} if there is a $dF$-invariant dominated splitting $E^s\oplus E^c\oplus E^u$ where $dF_{|E^s}$ is uniformly contracting, $dF_{|E^u}$ is uniformly expanding, and $E^c$ is non-trivial and non-hyperbolic. We say that $E^c$ is the {\emph{central bundle.}} The set $\Lambda$ is called {\emph{strongly partially hyperbolic}} if the three bundles $E^s$, $E^c$, and $E^u$ are non-trivial.
    See~\cite[Definition B.1]{BonDiaVia:05} for more details.
}\end{defi}

\begin{defi}[Porcupines]\label{def:porc}{\rm
    We call a compact $F$-invariant set $\Lambda$ of a (local) diffeomorphism $F$ a \emph{porcupine-like horseshoe} or \emph{porcupine} if
\begin{itemize}
    \item $\Lambda$ is the maximal invariant set in some neighborhood, {\emph{transitive}} (existence of a dense orbit), and strongly partially hyperbolic with
    one-dimensional central bundle,
    \item there is a subshift of finite type $\sigma\colon \Sigma \to \Sigma$ and a semiconjugation $\pi \colon \Lambda\to \Sigma$ such that $\sigma \circ \pi= \pi \circ F$, $\pi^{-1}(\xi)$ contains a continuum for uncountably many $\xi \in \Sigma$ and a single point for uncountably many $\xi\in\Sigma$.
\end{itemize}
We call $\pi^{-1}(\xi)$ a  \emph{spine} and say that it is {\emph{non-trivial}} if it contains a continuum. The spine of a point $X\in \Lambda$ is the set $\pi^{-1}(\pi(X))$.
}\end{defi}

For examples resemembling porcupines but having non-trivial spines only we refer to~\cite{BonDia:96, GorIly:99, GorIly:00}. Concerning bony attractors, according to the definition in~\cite{Kud:10} such an attractor may have only non-trivial fibers. Furthermore, there are quite interesting examples in~\cite{Kud:10,Kud:10b} of bony attractors where the trivial fibers form a ``graph'' of a continuous function over a subset of the shift space.

We are in particular interested in the case where $\sigma$ is the full shift defined on $\Sigma_2 =\{0,1\}^\ZZ$ and $\Lambda$ is a homoclinic class. More precisely, we have the following standard definition:

\begin{defi}[Homoclinic class]\label{def:homcla}
    {\rm
    Given a diffeomorphism $F$, the \emph{homoclinic class} $H(P,F)$ of a saddle point $P$ of $F$ is defined to be the closure of the transverse intersections of the stable and unstable manifolds of the orbit of $P$. Two saddle points $P$ and $Q$ are said to be \emph{homoclinically related} if the invariant manifolds of their orbits meet cyclically and transversally. We say that a homoclinic class is \emph{non-trivial} if it contains at least two different orbits.

    Given a neighborhood $U$ of the orbit of $P$, we call the closure of the set of points $R$ that are in the transverse intersections of the stable and unstable manifolds of the orbit of $P$ and have an orbit entirely contained in $U$ the \emph{homoclinic class relative to $U$}. We denote this set by $H_U(P,F)$.
    }
\end{defi}

\begin{rema}\label{r.funda}{\rm
	Homoclinically related saddles have the same index. Remark also that the homoclinic class of a saddle may contain periodic points that
are not homoclinically related to it. Indeed, this is the situation analyzed in this paper. Finally, observe that the homoclinic class $H(P,F)$ coincides with the closure of all saddle points that are homoclinically related to $P$. Moreover, a homoclinic class is always transitive. Finally, a non-trivial homoclinic class is always uncountable.
}\end{rema}

\begin{defi}[Lyapunov spectrum and gaps]\label{def:spectrum}
{\em Consider a compact invariant set $\Lambda$ of a diffeomorphism $F$ with a partially hyperbolic splitting $E^s\oplus E^c\oplus E^u$. Given a Lyapunov regular point $S\in \Lambda$, its \emph{Lyapunov exponent associated to the central direction $E^c$}
is
\begin{equation}\label{d.deflyap}
    \chi_c(S) \eqdef \lim_{n\to\infty}
    \frac{1}{n}\log \,\lVert dF^n|_{E^c_S}\rVert.
\end{equation}
We consider the \emph{spectrum of central Lyapunov exponents} of  $\Lambda$ defined by
\[
      \cL_{\rm reg}^c(\La)
      \eqdef  \big\{ \chi_c(S)\colon
      		S\in \La\text{ and }S \text{ is Lyapunov regular}\big\}.
\]
We say that $(\rho,\rho')$ is a {\emph{gap}}  of the spectrum of $\La$ if there are numbers  $\lambda\le\rho<\rho'\le\beta$ such that
\[
	(\rho,\rho')\cap\cL_{\rm reg}^c(\La)=\emptyset
	\quad\text{ and }\quad
	\lambda,\beta \in \cL_{\rm reg}^c(\La).
\]	
}
\end{defi}

The following is our main result.

\begin{theo} \label{t.bonyexample}
	There are $C^1$ local diffeomorphisms $F$ having a porcupine $\Lambda_F$ with the following properties:
	\begin{enumerate}
	    \item[(A)]
	        There is a continuous semiconjugation $\varpi\colon \Lambda_F\to \Si_2$, $\sigma\circ \varpi=\varpi \circ F$, such that
	    \begin{enumerate}
	        \item \label{i.segments}
	        There is an uncountable and dense subset of sequences $\xi\in \Si_2$ such that $\varpi^{-1} (\xi)$ is non-trivial. There is a residual subset of sequences $\xi\in\Si_2$ such that $\varpi^{-1} (\xi)$ is trivial.
	        \item \label{i.points}
	        There is an uncountable and dense subset of $\Lambda_F$ with non-trivial spines.
    	\end{enumerate}
    	\smallskip
	\item[(B)]	
		The subset of saddles of index $u+1$ of $\Lambda_F$ is dense in $\Lambda_F$. Moreover, $\Lambda_F$ contains also infinitely many saddles of index $u$ and thus is not uniformly hyperbolic.
	\smallskip
	 \item[(C)]
	 	The are numbers $0<\rho<\rho'$ such that $(\rho,\rho')$ is a gap of the spectrum of central Lyapunov exponents of $\Lambda_F$. Moreover, this spectrum contains an interval with negative and positive values. Furthermore, the pressure function $t\mapsto P(-t\log\,\lVert dF|_{E^c}\rVert)$ is not differentiable at some point, that is, has a first order phase transition.
    \end{enumerate}

    Furthermore, there is an open set $U$ such that $\Lambda_F$ is the (relative) homoclinic class $H=H_U(R,F)$ of a saddle $R$ of index $u+1\ge2$ satisfying:
    \begin{enumerate}
    	\item[(D)]
        The set $H=\Lambda_F$ is the locally maximal invariant set in $U$. Moreover, this class contains the (relative) non-trivial homoclinic class of a saddle of index $u$. Further, there is a saddle $Q\in H$ of index $u+1$ such that $H_U(Q,F)=\{Q\}\subset H$.
    \end{enumerate}
\end{theo}

Our examples are associated to step skew-product diffeomorphisms, locally we have
$$
	F(\widehat x, x)=
	\big( \Phi(\widehat x), f_{\widehat x}(x) \big), \quad
	\widehat x\in [0,1]^n, \, x\in [0,1],
$$
where $\Phi$ is a horseshoe map and  $f_{\widehat x}=f_0$ or $f_1$ for some injective maps $f_0$ and $f_1$ of  $[0,1]$. We observe that for our analysis we require only $C^1$ smoothness, that is weaker than the often required $C^{1+\varepsilon}$ hypothesis, and we base our proofs on a tempered distortion argument. Any step skew-product diffeomorphism with $C^\infty$ fiber maps $f_0$, $f_1$ which satisfy the below stated properties will provide an example for Theorem~\ref{t.bonyexample}.

Concerning the structure of spines, we have that non-trivial spines are tangent to the central direction $E^c$ and that spines of saddles of index $u+1$ are non-trivial and dense in $\Lambda_F$.
We also observe that, given any periodic sequence $\xi\in\Sigma_2$, there is a periodic point $P_\xi$ in $\varpi^{-1}(\xi)$. Under some mild additional Kupka Smale-like hypothesis, the spine of any saddle of index $u+1$ also contains saddles of index $u$. In fact, in this case, for every periodic sequence $\xi$ there is a saddle $S\in \Lambda_F$ of index $u$ projecting to $\xi$, $\varpi(S)=\xi$, (see Theorem~\ref{t.KS}).

\begin{ques} \label{q.allinterval}
    Do there exist examples of porcupine-like transitive
    sets such that $\varpi^{-1}(\xi)$ contains a continuum for an ``even larger'' subset  of $\Si_2$? Here larger could mean, for instance, a residual subset of $\Sigma_2$
    or a set of large dimension and we would like to state this question in a quite vague sense.
\end{ques}

Let us observe that the examples in~\cite{Kud:10} the set with non-trivial fibers is ``small'', though his setting is slightly different from ours.

\begin{figure}
\includegraphics[height=4cm]{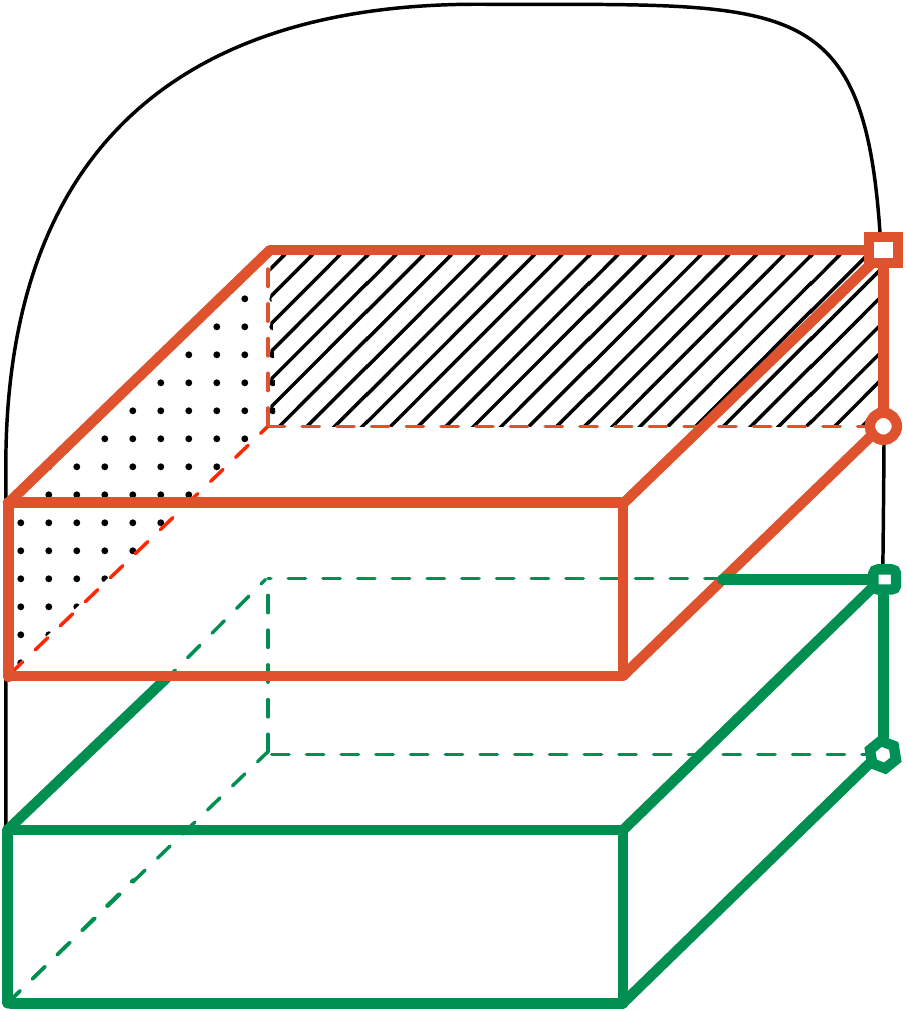}
$\displaystyle \overset{F}{\longrightarrow}$
\includegraphics[height=4.2cm]{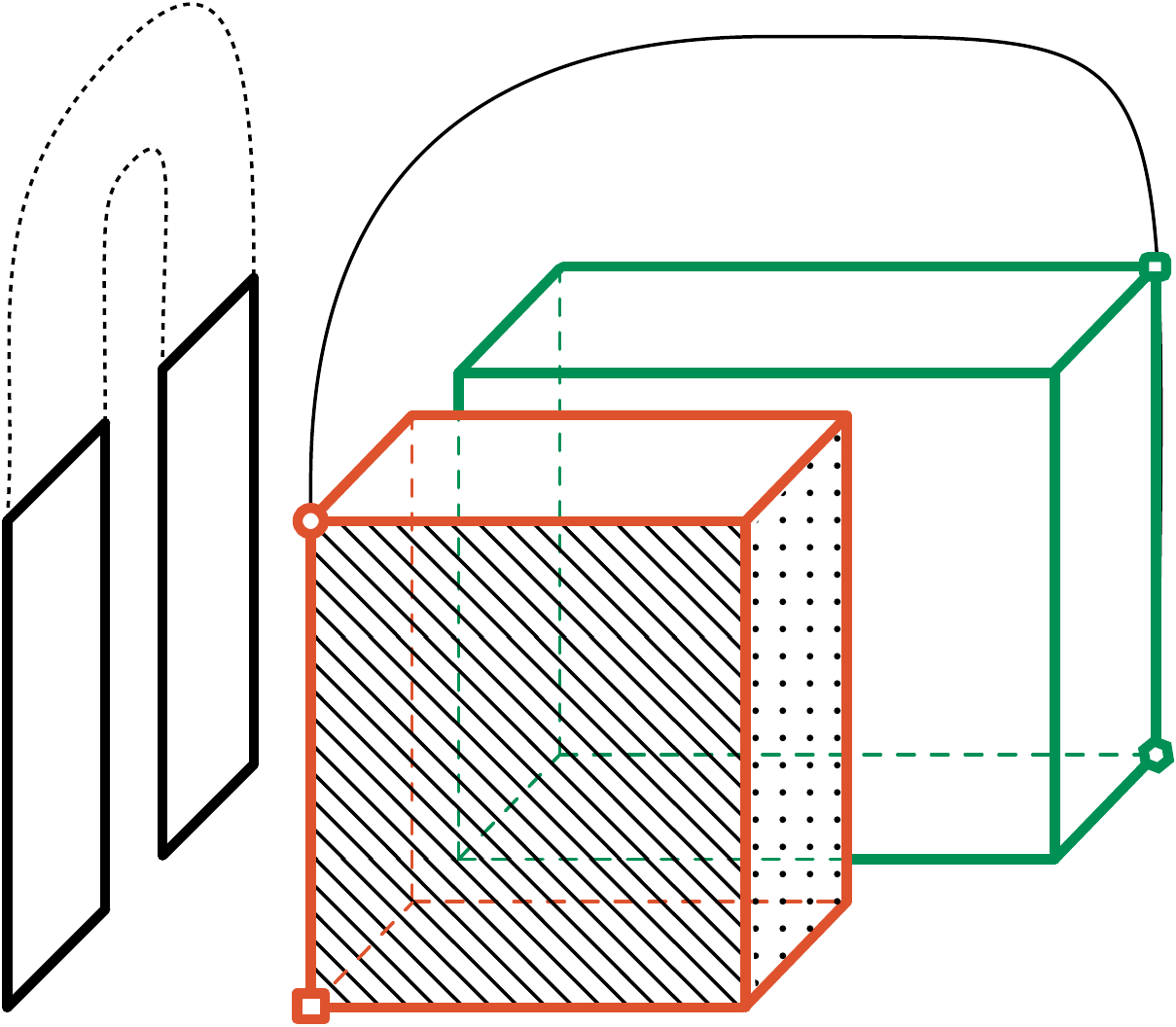}
\caption{Construction of a porcupine}
\label{fig.boxes}
\end{figure}

This paper is organized as follows. In Section~\ref{s.example} we describe the construction of our examples and derive first preliminary properties. In Section~\ref{s.onedim} we collect properties of the IFS generated by the interval maps $f_0$ and $f_1$.
In Section~\ref{s.tranhomint} we prove that the porcupine is a (relative) homoclinic class of a saddle of index $u+1$ and that it contains a non-trivial homoclinic class of a saddle of index $u$. This fact implies the porcupine is a transitive non-hyperbolic set containing infinitely many saddles of both types of indices. In this section we will systematically use the results in Section~\ref{s.onedim}. The skew-product structure allows us to translate properties of the IFS to the global dynamics. We will also study some particular cases that imply stronger properties.
In Section~\ref{sec:lyap} we finally study the Lyapunov exponents that are associated to the central direction. Note that our methods of proof in Sections~\ref{s.tranhomint} and~\ref{sec:lyap} are based on those used previously in studying heterodimensional cycles and homoclinic classes, see
for example~\cite{Dia:95,AbdBonCroDiaWen:07,BonDia:08,DiaHorRioSam:09}.
We conclude the proof of Theorem~\ref{t.bonyexample} in Section~\ref{sec:ptheo}.

%--------------------------------------------------------------------------------------------------
\section{Examples of porcupine-like homoclinic classes} \label{s.example}
%--------------------------------------------------------------------------------------------------

In this section we are going to construct examples of porcupine-like homoclinic classes satisfying the properties claimed in Theorem~\ref{t.bonyexample}.

Consider $s$, $u\in \NN$, the cube $\widehat \CC=[0,1]^{s+u}$, and a diffeomorphism $\Phi$ defined on $\RR^{s+u}$ having a horseshoe $\Ga$ in $\widehat\CC$ conjugate to the full shift $\sigma$ of two symbols and whose stable bundle has dimension $s$ and whose unstable bundle has dimension $u$.
Denote by  $\varpi\colon \Ga \to \Si_2$ the conjugation map
$\varpi \circ \Phi=\sigma \circ \varpi$. We consider the sub-cubes
$\widehat \CC_0$ and $\widehat \CC_1$ of $\widehat \CC$ such that
$\Phi$ maps each cube $\widehat \CC_i$ in a Markovian way into
$\widehat \CC$, where the cube $\widehat \CC_i$ contains all the
points $X$ of the horseshoe whose $0$-coordinate $(\varpi (X))_0$
is $i$. In order to produce the simplest possible example, we will
assume that $\Phi$ is affine in $\widehat\CC_0$ and
$\widehat\CC_1$.

\begin{defi}[The map $F$]\label{d.defF}
{\rm Let $\CC= \widehat \CC\times [0,1]$. Given a point $X\in \CC$,
we write $X=(\widehat x,x)$, where $\widehat x\in \widehat \CC$ and
$x\in [0,1]$. We consider a map
$$
F\colon \widehat \CC \times [0,1]\to \RR^{s+u}\times \RR
$$
given by
$$
F(\widehat x,x) \eqdef
 \begin{cases}
    (\Phi(\widehat x), f_0(x)) & \mbox{ if }X\in \widehat \CC_0\times [0,1],\\
    (\Phi(\widehat x), f_1(x)) & \mbox{ if }X\in \widehat\CC_1\times [0,1],
\end{cases}
$$
where $f_0$, $f_1 \colon [0,1]\to [0,1]$ are assumed to be $C^1$
injective interval maps satisfying the following properties (see Figure~\ref{fig.ifs}):
\begin{itemize}
\item[(F0.i)]
    The map $f_0$  is increasing and has exactly two hyperbolic fixed
    points, the point $0$ (repelling) and the point $1$ (attracting).
    Let $f_0^\prime (0)=\be>1$ and $f_0^\prime (1)= \la \in
    (0,1)$.
    Moreover, $\la\le f_0^\prime(x)\le\beta$ for all $x \in [0,1]$.
    \\[-0.3cm]
\item[(F0.ii)]
    There are fundamental domains $I_0=[a_0,b_0]\subset(0,1)$, $b_0=f_0(a_0)$, and $I_1=[a_1,b_1]$, $b_1=f_0(a_1)$,  of the map $f_0$ together with numbers $\alpha>1$ and $N\ge 1$ such that
    \[
    f_0^N(I_0)=I_1
    \quad\text{ and }\quad
    \la\cdot (f_0^N)^\prime (x)>\alpha>1
    \quad\text{ for all }x\in I_0.
    \]
    Moreover, $f_0$ is expanding in $[0,b_0]$ and contracting in $[a_1,1]$.
    \\[-0.3cm]
\item[(F1.i)]
    The map $f_1$ is a decreasing contraction satisfying
    \[
    \gamma' \eqdef \min\left\{ \lvert f_1^\prime(x)\rvert \colon x\in [0,1]\right\}\le
    \gamma\eqdef \max\left\{ \lvert f_1^\prime(x)\rvert \colon x\in [0,1]\right\}<1.
    \]
\item[(F1.ii)]
    We have
    \[
     |f_1^\prime(x)|\ge \overline\alpha >1/\alpha
     \quad\text{ for all }
     x\in [f_1^2(a_1),a_1].
     \]
\item[(F01)]
    The following conditions are satisfied
    \begin{enumerate}
    \item $f_1(1)=0$,
    \item $ f_1([a_1,1])\subset[0,a_0)$.
    \item $[0,f_0^{-2}(b_0))\subset f_1([0,1])$.
    \end{enumerate}
\end{itemize}
}\end{defi}

\begin{figure}[h]
\begin{minipage}[c]{\linewidth}
\centering
\begin{overpic}[scale=.35]{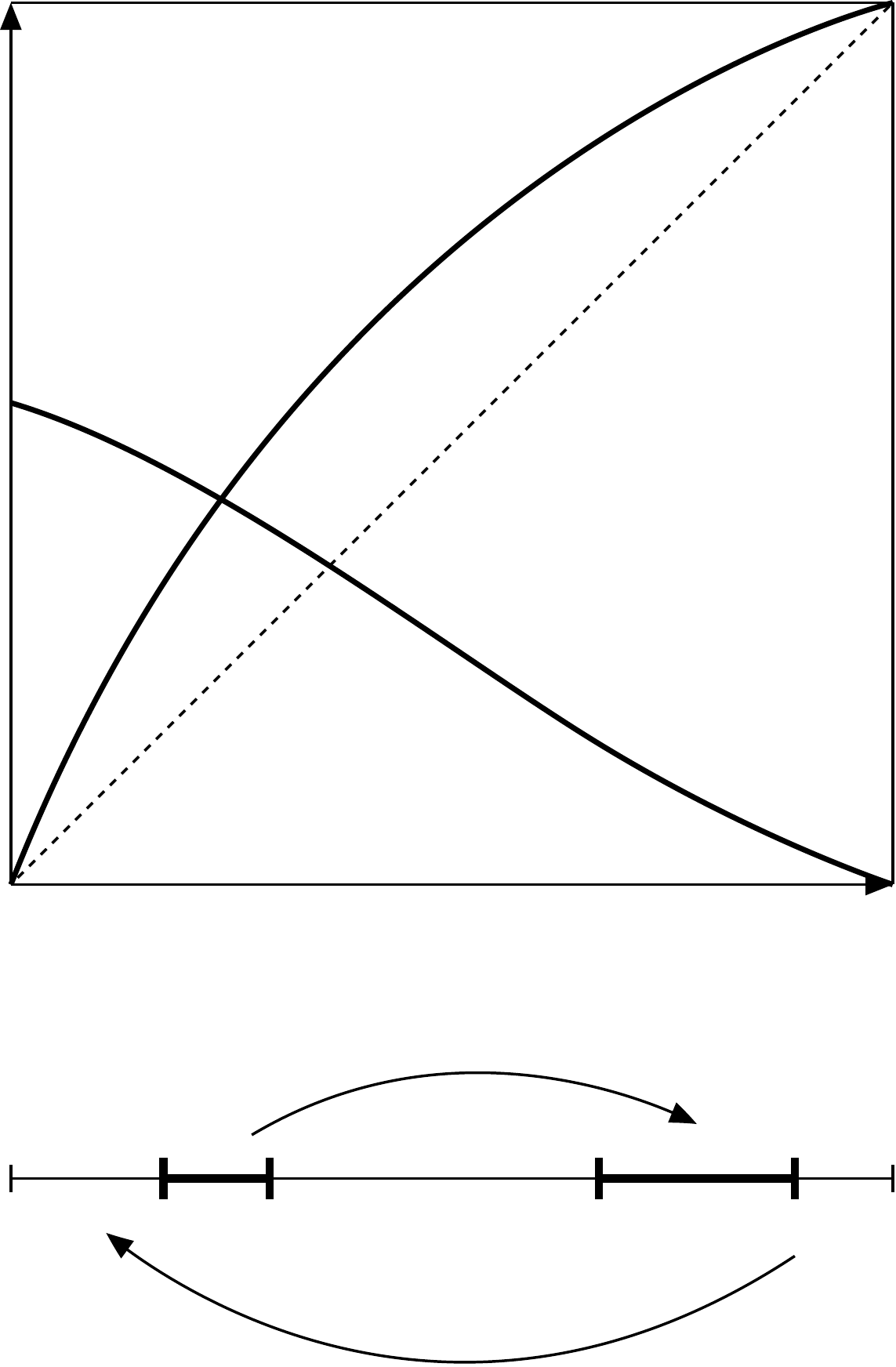}
    \put(30,86){$f_0$}	
     \put(38,52){$f_1$}
      \put(31,3){$f_1$}	
      \put(31,25){$f_0^N$}	
    \put(13,18){$I_0$}	
    \put(53,18){$I_1$}
    \put(0,29){$0$}
    \put(64,29){$1$}
    \end{overpic}
\end{minipage}
\caption{Iterated function system satisfying {(F0)}, {(F1)}, and~{(F01)}}
\label{fig.ifs}
\end{figure}
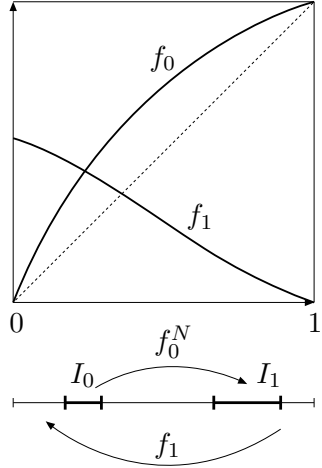

Note that in order to get the conditions above we need to require that
$$
\la\, \frac{1-\la}{1-\be^{-1}}>1.
$$

The \emph{maximal invariant set} of $F$ in the cube $\CC$ is defined by
\begin{equation}\label{e.Lambdapm}
    \La_F\eqdef\La_F^+\cap\La_F^-, \quad\text{ where }\quad
    \La_F^\pm\eqdef\bigcap_{i\in \NN} F^{\pm i}(\CC).
\end{equation}

\begin{rema}{\rm
	We point out that we restrict our analysis to the dynamics within the cube $\CC$. Notice that the usual definition of a (locally) maximal invariant set $\Lambda$ with respect to $F$ requires that $F$ is well-defined in some neighborhood $U$ of $\Lambda$ and that $\Lambda =\bigcap_{i\in\ZZ}F^i(U)$. Observe that in our case we can consider an extension of the local diffeomorphism $F$ to some neighborhood of $\CC$ such that $\Lambda_F$ is the locally maximal invariant set with respect to such an extension. Indeed, this can be done since the extremal points $P$ and $Q$ are hyperbolic.
	
	From now on we restrict our considerations to the dynamics in $\CC$. In particular, we consider relative homoclinic classes in $\CC$. For notational simplicity, we suppress  the dependence on $\CC$ and simply write $H(R,F)$.
}\end{rema}

Note that, by construction, for any saddle $ Q^\ast\in\CC$ the homoclinic class $H( Q^\ast,F)$ is contained in $\Lambda_F$ but, in principle, may be different from $\Lambda_F$. The analysis of the dynamics of $F|_{\Lambda_F}$ will be completed in Section~\ref{s.tranhomint}.
\smallskip

For simplicity, let us assume that the rate of expansion of the
horseshoe is stronger than any expansion of $f_0$ and $f_1$, that
is, in particular, stronger than $\beta$ and let us assume that
the rate of contraction of the horseshoe is stronger than any
contraction of $f_0$ and $f_1$, that is, in particular, stronger
than $\min\{\lambda,\gamma'\}$.
 In this way the $DF$-invariant splitting
$E^\sst\oplus E^c\oplus E^\uut$ defined over $\Lambda_F$ and given
by
\begin{equation}\label{e.splitt}
    E^\sst\eqdef\RR^s\times\{0^u,0\}, \quad
    E^c\eqdef\{0^s,0^u\}\times\RR, \quad
    E^\uut\eqdef\{0^s\}\times\RR^u\times\{0\}
\end{equation}
is dominated. Note that this splitting is $DF$-invariant because
of the skew-product structure of $F$.

The following is a key result in our constructions. Its proof will be completed in
Section~\ref{sec:ptheo}.

\begin{prop}\label{pro:spines}
    There is a periodic point $ Q^\ast\in \Lambda_F$ of index $u+1$ whose homoclinic class $H( Q^\ast,F)$ is a porcupine-like set having all the properties claimed in Theorem~\ref{t.bonyexample}.
\end{prop}

Let us now introduce some more notation and derive some simple
properties that can be obtained from the above definitions.

\begin{nota}\label{n.periodic}
{\rm Let us consider the sequence space $\Sigma_2=\{0,1\}^\ZZ$ and
adopt it with the usual metric
$d(\xi,\eta)=\sum_{i\in\ZZ}2^{-\lvert
i\rvert}\lvert\xi_i-\eta_i\rvert$ for
$\xi=(\ldots\xi_{-1}.\xi_0\xi_1\ldots)$,
$\eta=(\ldots\eta_{-1}.\eta_0\eta_1\ldots)\in\Sigma_2$. We denote
by $\xi=(\xi_0\ldots\xi_{m-1})^\ZZ$ the
 periodic sequence of period $m$ such that
$\xi_i=\xi_{i+m}$ for all $i$. We will always refer to the least
period of a sequence.
The zero sequence with $\xi_i=0$ for all $i$ we denote by
$0^\ZZ$. Further, we denote by $\xi=(0^\NN. 10^\NN)$ the sequence
that satisfies $\xi_0=1$, $\xi_{\pm i}=0$ for all $i\ne 0$. }
\end{nota}

Let $\theta=\varpi^{-1}(0^\ZZ)$ be the fixed point of $\Phi$ which
corresponds to the zero sequence $0^\ZZ$. Note that $\theta=0^{s+u}$.
Simplifying representation, we also assume that $[0,1]^s\times \{0^u\}=W^\st_\loc
(\theta,\Phi)$ and $\{0^s\}\times [0,1]^u=W^\ut_\loc
(\theta,\Phi)$.
Let us define
\begin{equation}\label{e.defPQ}
    P\eqdef(\theta,1)\quad\text{ and }\quad
    Q\eqdef(\theta,0).
\end{equation}
These saddles have indices $u$ and $u+1$, respectively. The
previous assumptions and the choice of $f_0$ imply immediately
that
\begin{equation}\label{e.manifolds}
\begin{split}
    [0,1]^s\times \{0^u\} \times (0,1]\subset W^\st (P,F)
    ,\\
    \{0^s\} \times [0,1]^u \times \{1\} \subset W^\ut  (P,F)
    ,\\
    [0,1]^s\times \{0^u\} \times \{0\} \subset W^\st (Q,F)
    ,\\
    \{0^s\}\times [0,1]^u\times  [0,1) \subset W^\ut (Q,F).
\end{split}
\end{equation}
In what follows we write
\[
	W^\st_\loc (Q,F)= [0,1]^s \times \{(0^s,0)\}
	\quad\mbox{and} \quad
	W^\ut_\loc (P,F)= \{0^s\} \times [0,1]^u\times \{1\} .
\]

\begin{defi}[Heterodimensional cycle]\label{def:hetcyc}
{\rm A diffeomorphism $F$ is said to have a
\emph{heterodimensional cycle} associated to saddle points $P$ and
$Q$ of different indices if their invariant manifolds
\emph{intersect cyclically}, that is, if $W^\st(P,F)\cap
W^\ut(Q,F)\ne\emptyset$ and $W^\ut(P,F)\cap
W^\st(Q,F)\ne\emptyset$. Here we denote by $W^\st(P,F)$
($W^\ut(P,F)$) the stable (unstable) manifold of the orbit of $P$
with respect to $F$. }
\end{defi}

The definition of $F$ immediately implies the following fact.

\begin{lemm}[Heterodimensional cycle]\label{r.periodicpoints}
    The points $P$ and $Q$ defined in~\eqref{e.defPQ} are saddle
    fixed points with indices $u$ and $u+1$, respectively,
    that are related by a heterodimensional cycle.
\end{lemm}

\begin{figure}[h]
\begin{minipage}[c]{\linewidth}
\centering
\begin{overpic}[scale=.45]{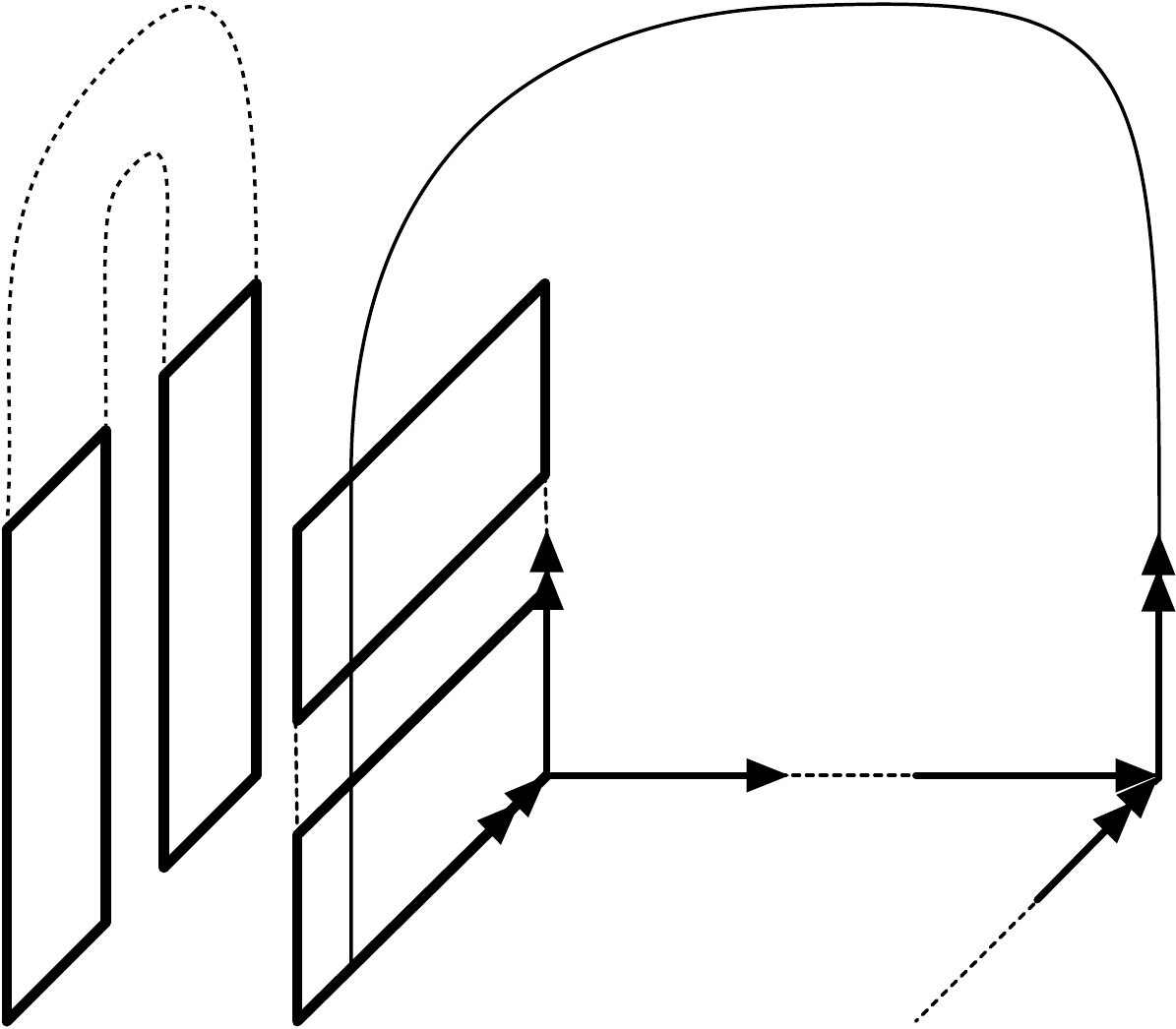}
    \put(37,22){$\widehat\CC_0$}	
    \put(37,47.5){$\widehat\CC_1$}
    \put(49,24.5){$Q$}
    \put(101,25){$P$}
    \end{overpic}
\end{minipage}
\caption{Heterodimensional cycle}
\label{fig.map_cyclic}
\end{figure}
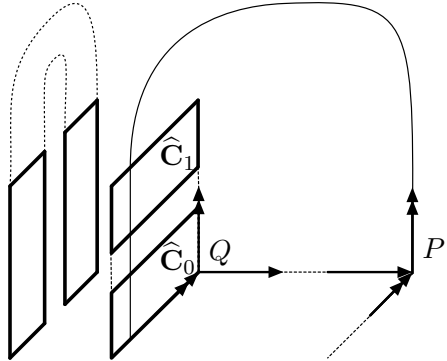

\begin{proof}
    Note that by~\eqref{e.manifolds} we have
    \[
    \{(0^s,0^u)\}\times (0,1)\subset W^\st (P,F) \cap W^\ut (Q,F).
    \]
    On the other hand, as $f_1(1)=0$, we have
    \[\begin{split}
    	F\big(W^u_{\rm loc}(P,F)\big)
	&= F\big(  (\{0^s\} \times [0,1]^u \times \{1\} ) \cap \CC_1\big)\\
	&=\{\theta^s\} \times [0,1]^u  \times \{0\} \subset W^{\ut}(P,F),
    \end{split}\]
    where $\theta^s=\varpi^{-1}(0^\NN.10^\NN)$ and $(\theta^s,0^u,0)\in W^s(Q,F)$,
    and hence $W^s(Q,F)\cap W^u(P,F)\ne \emptyset$. This gives a heterodimensional cycle associated to $P$ and $Q$, proving the lemma.
\end{proof}

We will now derive some properties of the homoclinic class $H(P,F)$.

\begin{lemm}\label{l.hp}
    The homoclinic class $H(P,F)$ contains the saddle $Q$. Therefore, this class is non-trivial and non-hyperbolic.
    Moreover, there are points in $\{(0^s,0^u)\}\times (0,1)$ that are contained in $H(P,F)$.
\end{lemm}

\begin{proof}
    Let $\widehat x=(x^s,0^u)=\varpi^{-1}(0^\NN . 10^\NN)$ and  $\Phi(\widehat x)=(y^s,0^u)$.
    Note that by Definition~\ref{d.defF} we have
    \begin{equation}\label{e.ys}
F \big( (\{x^s\} \times [0,1]^u\times \{0\}) \cap \CC_1 \big)=
    \{y^s\} \times [0,1]^u \times \{f_1(0)\} \subset W^u(P,F)
    \end{equation}
  and therefore, by \eqref{e.manifolds} and since $f_1(0)\in (0,1)$, $(y^s,0^u,f_1(0))$ is a transverse homoclinic point of $P$. This implies that $H(P,F)$ is non-trivial. Moreover, it implies that $W^u(P,F)$ accumulates at $W^u_{\rm loc}(P,F)$ from the left and thus the point $X=(x^s,0^u,0)$ is accumulated by a sequence $X_i\eqdef (x_i^s,0^u,x_i)$, $x_i=f_1(f_0^{N+i}(f_1(0)))>0$, of transverse homoclinic points of $P$ from the right.
  Finally, as $X\in W^s(Q,F)$ and $H(P,F)$ is invariant, we have
    $Q\in H(P,F)$.

    To prove that $\{(0^s,0^u)\}\times (0,1)$ contains points of $H(P,F)$, for each $\delta>0$
    and each (closed) fundamental domain $D$ of $f_0$ in $(0,1)$
consider the disk
\[
    D_\delta\eqdef [0,\delta]^s\times \{0^u\}\times D\subset W^s(P,F).
\]
    Note that for large $j$ the set $F^{-j}(D_\delta)$ contains some point $X_i$. Therefore, $D_\delta$ contains a
    transverse homoclinic point of $P$. As this holds for any $\delta>0$ and since $H(P,F)$ is a closed set, the
set $\{(0^s,0^u)\}\times D$ intersects $H(P,F)$ and the
    claimed property follows.
\end{proof}

Let us consider the attracting fixed point $\widehat p=f_1(\widehat p\,)\in (0,1)$ and denote by $({\widehat p\,}^s,{\widehat p\,}^u)=\varpi^{-1}(1^\ZZ)$ the corresponding fixed point for the horseshoe map $\Phi$. Note that $\widehat
P=({\widehat p\,}^s,{\widehat p\,}^u,\widehat p\,)$ is fixed with respect to $F$ and
has index $u$.

\begin{lemm}\label{l.hatP}
    The saddles $\widehat P$ and $P$ are homoclinically related.
\end{lemm}

\begin{proof}
    Let $\widehat P=({\widehat p\,}^s,{\widehat p\,}^u,\widehat p\,)$ and note that
    ${\widehat p\,}^s\times [0,1]^u\times \widehat p\subset W^u(\widehat P,F)$.
    Thus, together with $\widehat p\in(0,1)$ and $[0,1]^s\times\{0^u\}\times(0,1]\subset W^s(P,F)$,
    we obtain that $W^u(\widehat P,F)$ and $W^s(P,F)$ meet transversally.

    Let us now show that $W^s(\widehat P,F)$ and $W^u(P,F)$ also meet transversally.
    First note that $[0,1]\subset W^s(\widehat p,f_1)$ and therefore
    $[0,1]^s\times \{{\widehat p\,}^u\}\times [0,1]\subset W^s(\widehat P,F)$.
    Also note that by \eqref{e.ys}
we have that $\{y^s\}\times [0,1]^u\times\{f_1(0)\}\subset
W^u(P,F)$, where $f_1(0)$ in $(0,1)$.
    Thus $W^s(\widehat P,F)\pitchfork W^u(P,F)\ne\emptyset$, proving the assertion.
\end{proof}

\begin{rema}\label{r.barP}{\rm
    The construction in the proof of Lemma~\ref{l.hatP} shows also that any saddle
    $\overline P\in\Lambda_F$ of $F$ of index $u$ satisfies
    $W^u(\overline P,F)\pitchfork W^s(P,F)\ne\emptyset$ and $W^u(\overline P,F)\pitchfork W^s(\widehat P,F)\ne\emptyset$
}\end{rema}

As, by construction, $W^s_{\rm loc}(Q,F)$ does not intersect
$W^u(Q,F)\setminus\{Q\}$, we immediately obtain the following
fact.

\begin{lemm}\label{l.hq}
    We have $H(Q,F)=\{Q\}$.
\end{lemm}

%-------------------------------------------------------------------------------------------------------

%------------------------------------------------------------------------------------------
\section{One-dimensional central dynamics}\label{s.onedim}
%------------------------------------------------------------------------------------------

In this section we are going to derive properties of the abstract
iterated function system generated by the interval maps $f_0$ and
$f_1$ introduced in Section~\ref{s.example}. These properties will
carry over immediately to corresponding properties of the spines.
We point out that, in contrast to other commonly studied IFSs, in
our case the system is genuinely non-contracting and, in
particular, in Section~\ref{ss.expanding} we will study \emph{expanding
itineraries} of theses IFSs.

%------------------------------------------------------------------------------------------
\subsection{Iterated function system}\label{ss.onedim.ifs}
%------------------------------------------------------------------------------------------

Let us start with some notations.

\begin{nota}\label{n.cylinders}
{\rm
    Slightly abusing notation, for a given \emph{finite} sequence
    $\xi=(\xi_0\ldots \xi_m)$, $\xi_i\in\{0,1\}$, let
    $$
    f_{[\xi]}=
    f_{[\xi_0\ldots\xi_m]}\eqdef f_{\xi_m} \circ
    \cdots \circ f_{\xi_1}\circ f_{\xi_0} \colon [0,1]\to [0,1].
    $$
    Moreover, let
    $$
    \CC_{[\xi]}=
    \CC_{[\xi_0\dots \xi_m]} \eqdef \left\{X\in \CC
    \colon F^i(X)\in \CC_{\xi_i}\text{ for all }
    i=0,\dots,m\right\}.
    $$
    Given any set $K\subset \CC$, let
    $$
    K_{[\xi]}=
    K_{[\xi_0\dots\xi_m]}
    \eqdef K\cap \CC_{[\xi_0\dots\xi_m]}.
    $$
    Given a finite sequence $\xi=(\xi_{-m}\ldots\xi_{-1})$,  $\xi_i\in\{0,1\}$, we denote
    $$
    f_{[\xi.]}=
    f_{[\xi_{-m}\ldots\xi_{-1}.]}
    \eqdef   (f_{\xi_{-1}}\circ \ldots\circ f_{\xi_{-m}})^{-1}.
    $$
    Given a finite sequence $\xi=(\xi_{-m}\ldots\xi_{-1}.\xi_0\ldots \xi_n)$, $\xi_i\in\{0,1\}$, let
    $$
    f_{[\xi]}=
    f_{[\xi_{-m}\ldots\xi_{-1}.\xi_0\ldots\xi_n]}
    \eqdef  f_{[\xi_0\ldots\xi_n]} \circ f_{[\xi_{-m}\ldots\xi_{-1}.]}.
    $$
    Note that these maps are only defined on a closed subinterval of $[0,1]$.
    A sequence $\xi=(\ldots\xi_{-1}.\xi_0\xi_1\ldots)\in\Sigma_2$ is said to be
    \emph{admissible} for a point $x\in[0,1]$ if the map $f_{[\xi_{-m}\ldots\xi_{-1}.]}$
    is well-defined at $x$ for all $m\ge 1$.
    Note that admissibility of a sequence $\xi$ does not depend on the symbols $(\xi_0\xi_1\ldots)$.
}
\end{nota}

%------------------------------------------------------------------------------------------
\subsection{Expanding itineraries} \label{ss.expanding}
%------------------------------------------------------------------------------------------

We now start investigating  expanding behavior of the
iterated function system.

Recall that $I_0=[a_0,b_0]=[f_0^{-1}(b_0),b_0]$. Given a closed
interval $J\subset [f_0^{-2}(b_0),b_0]$, we start by localizing an itinerary for which the iterated function system is expanding. In what follows we will always assume that the closed intervals are non-trivial. Recall the definition of $a_1$ in {(F0.ii)} and let us define
\[
    n(J)\eqdef \min\left\{ n\ge 1 \colon f_0^n(J)\subset [a_1,1)\right\}.
\]
Now let $J^\prime \eqdef f_{[0^{n(J)}1]}(J)$ and observe that by
(2) in {(F01)} this interval is contained in $(0,a_0)$.  Let
\[
    m(J)\eqdef \min\left\{ m\ge 0 \colon f_0^m(J^\prime)\cap I_0
        \ne \emptyset\right\}.
\]
Note that, by our choice of fundamental domains, we have $m(J)\ge
1$ and either $n(J)=N$ or $N+1$ with $N$ given in {(F0.ii)}.

\begin{lemm}[Expanding itineraries]\label{l.expanding}
     There is a constant $\kappa>1$ such that for every closed interval $J\subset [f_0^{-2}(b_0),b_0]$
     and every $x\in J$ we have
     $$
        \big\lvert\big(f_{[0^{n(J)}1\,0^{m(J)}]}\big)'(x) \big\rvert\ge \kappa.
    $$
\end{lemm}

\begin{proof}
Recall that $n(J)= N$ or $=N+1$ and that $m(J)\ge 1$. Observe that the
hypotheses {(F0.i)}, {(F0.ii)}, and {(F1.ii)}  imply that for any
$x\in J$ we have
\[
\lvert\big(f_{[0^{n(J)}1\,0^{m(J)}]}\big)'(x) \big\rvert =\lvert
(f_0^{m(J)} \circ f_1 \circ f_0^{n(J)})'(x)\rvert \ge
\overline\alpha\, \frac{\alpha}{\lambda}\,\lambda >1,
\]
using the fact that $f_0^{m(J)}$ is applied to points in an
interval $[0,a_0]$, where $f_0'>1$. Taking $\kappa\eqdef
\overline\alpha \,\alpha$, this proves the lemma.
\end{proof}

\begin{defi}[Expanding successor]\label{d.expanding}
{\rm Given an interval $J\subset [f_0^{-2}(b_0),b_0]$,  we
associate to $J$ the finite sequence $\xi(J)$ given by
$$
\xi_0=\cdots =\xi_{n(J)-1}=0, \quad \xi_{n(J)}=1, \quad
\xi_{n(J)+1}=\cdots =\xi_{n(J)+m(J)}=0,
$$
where $n(J)$ and $m(J)$ are defined as above. In view of
Lemma~\ref{l.expanding}, we call $(\xi_0\ldots\xi_{n(J)+m(J)})$
the {\emph{expanding itinerary}} of $J$. We call the interval
\[
f_{[\xi_0\ldots\xi_{n(J)+m(J)}]}(J) \subset (0,1)
\]
the {\emph{expanded successor}} of $J$.
We say that an interval $J''$ is the {\emph{$i$-th expanded
successor of $J$}} if there is a sequence of intervals $J_0=J$,
$J_1$, $\ldots$, $J_{i-1}$, $J_i= J''$ such that for all $j=0$,
$\ldots$, $i-1$, we have
\[
    J_j\subset [f_0^{-2}(b_0),b_0] \quad\text{ and }\quad J_{j+1}
    \text{ is the expanded successor of }J_j.
\]
We denote the $i$-th expanded successor of $J$ by $J_{\langle
i\rangle}$.
Using this notation we have defined the expanded finite sequences $\xi_{\langle
i\rangle}=\xi(J_{\langle i\rangle})$ with
\begin{equation}\label{e.Ji}
    J_{\langle i+1\rangle}=f_{[\xi_{\langle i\rangle}]}(J_{\langle i\rangle}).
\end{equation}
We denote by $\lvert\xi_{\langle i\rangle}\rvert$ the length of
this sequence. }
\end{defi}

\begin{rema}\label{r.import}{\rm
    For future applications we remark that the numbers $n(J)$ and $m(J)$
    are both bounded from above by some number that is independent of the interval $J\subset [f_0^{-2}(b_0), b_0]$.
    Therefore, the definition of the expanded successor of an interval $J$ involves
    a concatenation of a number of maps $f_0$ and $f_1$ that is bounded by some constant that is independent on $J$. In particular, there are constants $\kappa_1$, $\kappa_2>0$ independent of $J$ such that for all $x\in J$ we have
    \[
	\kappa_1 \le |(f_{[0^{n(J)}1\,0^{m(J)}]})'(x)| \le \kappa_2.
    \]
}\end{rema}

In what follows we denote by $\lvert I\rvert$ the length of an interval $I$.

\begin{rema}\label{r.successor}
    {\rm
    Let $J$ be a closed subinterval in $[f_0^{-2}(b_0),b_0]$. By Lemma~\ref{l.expanding}, there is a constant $\kappa>1$ that is independent of the interval $J$ such that the expanded successor $J_{\langle 1\rangle}=f_{[\xi(J)]}(J)$ of $J$ satisfies
    $$
    |J_{\langle 1\rangle}|= |f_{[\xi(J)]}(J)|\ge \kappa\, |J|.
    $$
    Moreover, by definition of $m(J)$, the interval $J_{\langle 1\rangle}$ intersects $[f_0^{-1}(b_0),b_0]$.
}
\end{rema}

The following lemma is the main result of this subsection.

\begin{lemm}\label{l.successor}
	Given a closed interval $J\subset [f_0^{-2}(b_0),b_0]$, there is a number $i(J)\ge 1$ such that the $j$-th expanded successor $J_{\langle j\rangle}$ of $J$ is defined
    for all $j=1$, $\dots$, $i(J)-1$ and  that $J_{\langle i(J)\rangle}$ contains the
     fundamental domain $[f_0^{-2}(b_0),f_0^{-1}(b_0)]$.
\end{lemm}

\begin{proof}
    Note that the expanded successor is defined for any interval $J\subset (f^{-2}(b_0),b_0]$. Assume, inductively, that for all $j=0$, $\dots$, $i$ the $j$-th expanded successor
    $J_{\langle j\rangle}$ of $J=J_{\langle 0\rangle}$ is defined and that
    $J_{\langle j\rangle}\subset [f_0^{-2}(b_0),b_0]$.
    Then the $(i+1)$-th expanded successor $J_{\langle i+1\rangle}$ of $J$ is also defined.

    Since for every $j=0$, $\ldots$, $i$ the interval $J_{\langle j+1\rangle}$ is the
    successor of $J_{\langle j\rangle}$,
    by Lemma~\ref{l.expanding} we have
    $$
        \lvert J_{\langle i+1\rangle}\rvert
            \ge \kappa^{i+1} \lvert J_{\langle 0\rangle}\rvert.
    $$
    Since the size of $[f_0^{-2}(b_0),b_0]$ is bounded there is a first $i(J)$ such that
    $J_{\langle 0\rangle},J_{\langle 1\rangle}$, $\dots$, $J_{\langle i(J)\rangle}$ are
    defined and $J_{\langle i(J)\rangle}$ is not contained in $[f_0^{-2}(b_0),b_0]$.
    Since by Remark~\ref{r.successor} the interval $J_{\langle i(J)\rangle}$
    intersects $[f_0^{-1}(b_0),b_0]$, this implies that
    \[
    	[f_0^{-2}(b_0),f_0^{-1}(b_0)]\subset J_{\langle i(J)\rangle},
    \]
    from which the claimed property follows.
\end{proof}

We yield the following result that can be of independent interest.

\begin{prop}[Sweeping property]\label{l.expiti}
  Given a closed interval $H\subset (0,1)$, there is a finite sequence $(\xi_0\ldots\xi_n)$ so that $f_{[\xi_0\ldots\xi_n]}(H)$ contains the fundamental domain $[f_0^{-2}(b_0),f_0^{-1}(b_0)]$.
\end{prop}

\begin{proof}
	Just note that there exists a number $m\ge1$ so that $f_{[0^m1]}(H)$ is contained in $(0,f_0^{-2}(b_0))$. Therefore, there is $k\ge1$ so that $f_{[0^m10^k]}(H)$ contains an interval $J\subset (f_0^{-2}(b_0),b_0)$ to which we can apply Lemma~\ref{l.successor}.
\end{proof}

\begin{defi}[Expanding sequence]{\rm
	In view of Lemma~\ref{l.successor}, given a closed interval $J\subset [f_0^{-2} (b_0),b_0]$ we consider its (finite) {\emph{expanding sequence}} $\xi(J)$ obtained by concatenating the finite sequences $\xi_{\langle i\rangle}=\xi(J_{\langle
i\rangle})$ corresponding to  the expanding successors $\xi_{\langle 1\rangle}$, $\ldots$, $\xi_{\langle i(J)\rangle}$ of $J$.

Note that, by definition of $i(J)$, we have $[f_0^{-2}(b_0),b_0]\subset f_{[\xi(J)]} (J)$ and $\lvert(f_{[\xi(J)]})'(x)\rvert>1$ for all $x\in J$.
}\end{defi}

An immediate consequence of Lemma~\ref{l.successor} and the
previous comments is the following lemma.

\begin{lemm}\label{l.newfixedexpandingpoint}
Given a closed interval $J\subset [f_0^{-2}
(b_0),b_0]$ and its expanding sequence $\xi(J)$,
there is a unique expanding fixed point $ q^\ast_J\in J$ of
$f_{[\xi(J)]}$. Moreover,  $W^u( q^\ast_J,f_{[\xi(J)]})$
contains $[f_0^{-2} (b_0),f_0^{-1}(b_0)]$.
\end{lemm}

\begin{proof}
Observe that $J\subset [f_0^{-2} (b_0),b_0] \subset f_{[\xi(J)]}
(J)$ and that the map $f_{[\xi(J)]}$ is uniformly expanding in
$J$.
\end{proof}

%-------------------------------------------------------------------------------------------------------

%------------------------------------------------------------------------------------------
\subsection{Lyapunov exponents close to $0$}\label{ss.lyap}
%------------------------------------------------------------------------------------------

In this section we are going to construct fixed points (of
contracting and expanding type) with respect to certain maps
$f_{[\xi_0\ldots\xi_{m-1}]}$ whose Lyapunov
exponents are arbitrarily close to $0$. Here, given $p\in[0,1]$
and an admissible sequence
$\xi=(\ldots\xi_{-1}.\xi_0\xi_1\ldots)\in\Sigma_2$ of $p$, the
\emph{(forward) Lyapunov exponent} of $p$ with respect to the sequence $\xi$ is
defined by
\[
    \chi(p,\xi)\eqdef
    \lim_{n\to\infty}\frac{1}{n}\log \, \big\lvert (f_{[\xi_0\ldots\xi_{n-1}]})'(p)\big\rvert
\]
whenever this limit exists.
Otherwise we denote by $\underline\chi(p,\xi)$ and $\overline\chi(p,\xi)$ the \emph{lower} and the \emph{upper Lyapunov exponent} defined by taking the lower and the upper limit, respectively.

Given a periodic sequence $(\xi_0\ldots\xi_{m-1})^\ZZ\in\Sigma_2$ and a point $p_{(\xi_0\ldots\xi_{m-1})^\ZZ}=f_{[\xi_0\ldots\xi_{m-1}]}(p_{(\xi_0\ldots\xi_{m-1})^\ZZ})$,
we have
\begin{equation}\label{e.lyapexpo}
    \chi(p,(\xi_0\ldots\xi_{m-1})^\ZZ)=
    \frac 1 m \log\, \big\lvert (f_{[\xi_0\ldots\xi_{m-1}]})'(p_{(\xi_0\ldots\xi_{m-1})^\ZZ})\big\rvert.
\end{equation}

We are going to prove the existence of periodic points of contracting and expanding
type with Lyapunov exponents arbitrarily close to $0$.

\begin{prop}\label{p.lajk}
    For every $\varepsilon>0$ there exists a finite sequence $(1^\ell0^m10^j)$
    such that the  map $f_{[1^\ell0^m10^j]}$ is uniformly contracting in $[0,1]$
    and its fixed point $p$ is attracting, has a Lyapunov exponent in $(-\varepsilon,0)$,
    and has a stable manifold $W^s(p,f_{[1^\ell0^m10^j]})$ that
    contains the interval $[0,1]$.
\end{prop}

\begin{prop}\label{p.posspec}
    For every $\varepsilon>0$ there exists a finite sequence $(\xi_0\ldots\xi_{n-1})$
    such that the  map $f_{[\xi_0\ldots\xi_{n-1}]}$ has an expanding fixed point whose
    Lyapunov exponent is in $(0,\varepsilon)$.
\end{prop}

Before proving the above two propositions, we formulate some
preliminary results.

%------------------------------------------------------------------------------------------
\subsubsection{Tempered distortion}
%------------------------------------------------------------------------------------------

First, we verify a distortion property. Note
that we establish the tempered distortion property that holds true
if $f_0$ is only a $C^1$ map, instead of focusing on a bounded
distortion property that would require the standard
$C^{1+\varepsilon}$ assumption to be satisfied.

We will say that an interval $J\subset (0,1)$ \emph{contains at
most $K$ consecutive fundamental domains} of $f_0$ if any orbit of
$f_0$ hits at most $K+1$ times this interval.

\begin{lemm}[Tempered distortion]\label{l.distoe}
    Given a point $\,\widehat p\in(0,1)$ and a number $K\ge 1$, there exists a positive
    sequence $(\rho_k)_{k\ge0}$ decreasing to $0$ such that for every interval $J$ containing
    $\widehat p$ and containing at most $K$ consecutive fundamental domains of $f_0$  we have
    \[
        e^{-k\rho_{\lvert k\rvert}} \frac{\lvert f_0^{\pm k}(J)\rvert}{\lvert J\rvert}
        \le (f_0^{\pm k})'(x)
        \le e^{k\rho_{\lvert k\rvert}} \frac{\lvert f_0^{\pm k}(J)\rvert}{\lvert J\rvert}
    \]
    for all $k\in\ZZ$ and for every $x\in J$.
\end{lemm}

\begin{proof}
    Let $x$, $y\in J$.
    As $f_0'$ is bounded away from $0$ and the map $y\mapsto\log y$ is Lipschitz if $y$ is bounded away from $0$, there exists some positive constant $c$ and a positive sequence $(\widetilde\rho_k)_k$
    decreasing to $0$ so that for every $k\ge 1$
    \[
        \left\lvert\log\frac{ (f_0^k)'(x)}{ (f_0^k)'(y)}\right\rvert
        \le c\sum_{n=0}^{k-1}\lvert f_0'(f_0^n(x)) - f_0'(f_0^n(y))\rvert
        \le c\sum_{n=0}^{k-1}\widetilde\rho_n.
    \]
    Here the latter estimate follows from continuity of $f_0'$ and the fact that
    $\lvert f_0'(x)-f_0'(y)\rvert\to0$ as $\lvert x-y\rvert\to0$ and the observation
    that $\lvert f^n_0(J)\rvert\to0$ as $n\to\infty$. Now take
    \[
        \rho_{k}\eqdef
        \frac{c\,(\widetilde\rho_0+\cdots+\widetilde\rho_{k-1})}{k}
    \]
    and note that $\rho_k\to0$ as $k\to\infty$. Taking $y$ such that
    $\lvert (f_0^k)'(y)\rvert = \lvert f^k(J)\rvert / \lvert J \rvert$
    we get the claimed property.

Analogously, we have $\lvert f^{-n}_0(J)\rvert\to 0$ as
$n\to\infty$, from which we can conclude the case $k\le0$.
\end{proof}

%------------------------------------------------------------------------------------------
\subsubsection{Looping orbits}
%------------------------------------------------------------------------------------------

We now show that the derivative along a looping orbit starting and
returning to a fixed fundamental domain growths only
sub-exponentially with respect to its length.

\begin{lemm}\label{l.boufun}
    Given a fundamental domain $J\subset (0,1)$, there exists a number $K\ge 1$
    such that for all $m\ge1$ sufficiently large the interval $f_{[0^m1]}(J)$
    contains at most $K$ consecutive fundamental domains of $f_0$.
\end{lemm}

\begin{proof}
    For each  $m\ge1$ let us define numbers $a_m\in(0,1)$ by
    \[
        f_0^m(J)=[1-a_m,f_0(1-a_m)].
    \]
    Note that, because $f_0(1)=1$, $f_0'(1)=\lambda$, the derivative $f_0'$ is
continuous, and $f_0^m(J)$ converges to $1$ (and thus $a_m\to 1$)
as $m\to\infty$, if $m$ is large enough, we have
\[
    \lvert f_0^m(J)\rvert
    = f_0(1-a_m)-(1-a_m) \approx f_0(1)-a_m f_0'(1)-(1-a_m)
    = a_m(1- \lambda)
\]
and hence
\begin{equation}\label{e.eqnumber}
    \frac 1 2  \, a_m \, (1-\lambda)
    \le
    \lvert f_0^m(J)\rvert
    \le 2 \, a_m \, (1-\lambda).
\end{equation}
Recalling the definitions of $\ga$ and $\ga'$ in (F1.i), one has that
\begin{equation}\label{e.etapa1}
\ga'\, \frac 1 2  \, a_m \, (1-\lambda)
    \le
    \lvert f_1( f_0^m(J))\rvert
  \le   \ga\,    2 \, a_m \, (1-\lambda).
\end{equation}
Similarly one obtains that there is a constant $C>1$ independent
of large $m$ such that $C^{-1}\,a_m \le f_1 (f_0(1-a_m))\le C\,
a_m$. Hence
\begin{equation}\label{e.etapa2}
	f_{[0^m1]}(J)
	\subset \big[C^{-1}\, a_m, C\,a_m(1+2\ga\,(1-\lambda))\big].
\end{equation}
Noting that the derivative of $f_0$ in $f_{[0^m1]}(J)$ is close to
$\be$ and bigger than some $\be'$ close to $\be$, we get that for
large  $m$ the interval $f_{[0^m1]}(J)$ contains at most $\ell+2$
fundamental domains where $\ell$ is the largest natural number with
$$
	(\be')^\ell \, C^{-1}\,a_m
	\le C\, a_m(1+2\ga\,(1-\lambda)).
$$
Notice that $a_m$ cancels and hence the number $\ell$ does not depend on $m$ if $m$ is large enough. This finishes the proof of the lemma.
\end{proof}

We will now use the above lemma to prove the following.

\begin{lemm}\label{l.looping}
Given a fundamental domain $J\subset (0,1)$ of $f_0$, there exists
$m_0\ge1$ and a positive sequence $(\rho_n)_n$ decreasing to $0$ so that for every $m\ge m_0$ there exists a number $j>0$
      such that the interval $f_{[0^m10^j]}(J)$
    intersects $J$. If $j=j(m)$ is the smallest positive number with this property then
    \[
        e^{-j\rho_j-m\rho_m}\le \lvert (f_{[0^m10^j]})'(x)\rvert \le e^{j\rho_j+m\rho_m}
    \]
     for every $x\in J$ and every $m\ge m_0$ and $j=j(m)$.
\end{lemm}

\begin{proof}
    Let $J=[a,f_0(a)]\subset (0,1)$ be a fundamental domain with respect to $f_0$.
    There exists a number $m_0\ge1$ so that for every $m\ge m_0$ and every $x\in J$ we have $f^{m}_0(x)>f_1^{-1}(a)$ and that $f_1 (f^{m}_0(x))$ is in the expanding region of $f_0$. Moreover, the interval $f_{[0^m10^j]}(J)$ intersects $J$ for some
$j>0$.

    As in the above proof, for each  $m\ge m_0$ let us denote \[f_0^m(J)=[1-a_m,f_0(1-a_m)].\]
    As for~\eqref{e.eqnumber} we obtain
\[
    \frac 1 2 \, a_m \, (1-\lambda)
    \le \lvert f_0^m(J)\rvert
    \le 2\, a_m \, (1-\lambda).
\]
The tempered distortion result in Lemma~\ref{l.distoe} now implies
that there is a positive sequence $(\widehat\rho_k)_k$ decreasing
to $0$ such that for all $x\in J$
\begin{equation}\label{e.star}
    \frac{1}{2} \, \frac{a_m \, (1-\lambda)}{\lvert J\rvert} \, e^{-m\widehat\rho_m}
    \le (f_0^m)'(x)
    \le 2\, \frac{a_m \, (1-\lambda)}{\lvert J\rvert}\, e^{m\widehat\rho_m}.
\end{equation}
We consider now the fundamental domain of $f_0$
\[
    L\eqdef
    \big[f_0^{-1}\big( f_1(1-a_m)\big),f_1(1-a_m)\big].
\]
Note that by Lemma~\ref{l.boufun} the interval $f_{[0^m1]}(J)$
contains at most $K$ fundamental domains. By our choice the right extreme of $L$ is the right extreme of $f_{[0^m1]}(J)$ and therefore
\[
f_{[0^m1]}(J)\subset \widetilde L \eqdef L \cup f_0^{-1}(L) \cup
\ldots \cup f_0^{-K}(L).
\]

In the next step we compare the lengths of $\widetilde L$ and $f_0^m(J)$.
Arguing exactly as above, using the fact that $f_0'(0)=\beta$,
$f_0'(x)\le \beta$, $\ga'\le \lvert f_1'\rvert \le \gamma$, and that
$f_1(1-a_m)$ is close to $0=f_1(1)$ for $m$ large enough, we
obtain
\begin{equation}\label{e.am}
    \frac 1 2 \,\ga \,a_m \, (1-\beta^{-1})
    \le \lvert L\rvert \le
    \ga'\, \, a_m \,  (1-\beta^{-1}).
\end{equation}
Since for $i\ge 0$ we have that $\be^{-i}\, \lvert L\rvert \le \lvert f_0^{-i}(L)\rvert
\le (\beta')^{-i} \,\lvert L\rvert$ for some number $1<\be'<\be$, from \eqref{e.am}
we immediately obtain constants $k_1,k_2>0$ such that
\begin{equation}\label{e.Ltilde}
k_1 \,a_m \le \lvert\widetilde L\rvert\le k_2 \, a_m.
\end{equation}
Let us fix constants $M^-$ and $M^+$ such that if $J'$, $J''$ are any two
non-disjoint fundamental domains of $f_0$ then
\begin{equation}\label{e.choice}
    M^-\lvert J'\rvert \le \lvert J''\rvert \le M^+\lvert J'\rvert.
\end{equation}
For large $m$ the  sets $f_{[0^m1]}(J)$ and $\widetilde L$ both are
to the left of $J$. Thus there exists a smallest positive integer
$j=j(m)$ such that $f_0^j(f_{[0^m1]}(J))$ (and thus $f_0^j(\widetilde L\,)$)
intersects $J$ for the first time (the same number for both
intervals). We now apply the tempered distortion property in
Lemma~\ref{l.distoe} to the interval $\widetilde L$. Hence, there
exists a sequence $(\widetilde \rho_k)_k$ decreasing to $0$ so that for
all $x\in \widetilde L$ we have
\[
    e^{-j\widetilde\rho_j}\, \frac{\lvert f_0^j( \widetilde L)\rvert}{\lvert \widetilde L\rvert}
    \le \lvert (f_0^j)'(x)\rvert
    \le e^{j\widetilde\rho_j} \, \frac{\lvert f_0^j(\widetilde L)\rvert}{\lvert
    \widetilde L\rvert}.
\]
The definition of $\widetilde L$ and \eqref{e.choice} imply that
$$
    M^- \, \lvert J\rvert
    \le \lvert f_0^j(\widetilde L)\rvert
    \le \widetilde M^+  \, \lvert J\rvert,
    \quad \mbox{where} \quad \widetilde M^+= \sum_{j=0}^K (M^+)^j.
$$
Thus, by the two previous equations, for $x\in \widetilde L$ we yield
\[
    e^{-j\widetilde\rho_j}\, \frac{ M^-\, \lvert J\rvert}{\lvert \widetilde L\rvert}
    \le \lvert (f_0^j)'(x)\rvert
    \le e^{j\widetilde\rho_j}\,\frac{\widetilde M^+ \, \lvert J\rvert}{\lvert \widetilde L\rvert}.
\]
This inequality together with~\eqref{e.Ltilde} imply that  for all $x\in
\widetilde L$ (and thus for all $x\in f_{[0^m1]}(J)$) we have
\begin{equation}\label{e.lahgs}
    \frac{e^{-j\widetilde\rho_j}\, M^- \, \lvert J\rvert}{k_2\,a_m}
    \le \lvert (f_0^j)'(x)\rvert
    \le \frac{e^{j\widetilde\rho_j}\, \widetilde M^+ \, \lvert J\rvert}{k_1\,a_m} .
\end{equation}
Now putting together~\eqref{e.star} and~\eqref{e.lahgs} and
recalling that $\ga'\le| f_1'|\le \gamma$, we see that the factors
$\lvert J\rvert$ and $a_m$ cancel. Hence we obtain for every
$x\in J$
\[
e^{-j\widehat \rho_j-m\widetilde\rho_m}\,
    \frac{(1-\lambda)\, \ga' \, M^-}{2\,k_2}
    \le \lvert (f_0^j\circ f_1\circ f_0^m)'(x)\rvert
     \le e^{j\widehat \rho_j+m\widetilde\rho_m}\,
        \frac{2\, (1-\lambda)\, \gamma \, \widetilde M^+}{k_1}.
\]
Thus there is some $\widetilde C>1$ independent of $m$ and
$j=j(m)$ such that
\[
    \widetilde C^{-1} \, \big( {e^{-j\widehat
    \rho_j-m\widetilde\rho_m}} \big)
    \le  \lvert (f_0^j\circ f_1\circ f_0^m)'(x)\rvert
    \le \widetilde C \,  \big( e^{j\widehat \rho_j+m\widetilde\rho_m}\big).
\]
The claimed property hence follows with  $\rho_n\eqdef
\max\{\widehat \rho_n,\widetilde\rho_n+\frac{1}{2\, n} \log
\widetilde C\}$.
\end{proof}

%------------------------------------------------------------------------------------------
\subsubsection{Weak contracting and expanding looping orbits}
%------------------------------------------------------------------------------------------

We are now ready to prove the above propositions.

\begin{proof}[Proof of Proposition~\ref{p.lajk}]
Recall that we denoted by $\widehat p\in(0,1)$ the attracting
fixed point of $f_1$. Consider a fundamental domain $J$ of $f_0$
containing $\widehat p$ in its interior and some $\ell_0\ge 1$
such that for all $\ell\ge\ell_0$ the interval $f_1^\ell([0,1])$
is contained in $J$. This is possible because $f_1^\ell([0,1])$
converges to $\widehat p$.

By Lemma~\ref{l.looping}, there exist a number $m_0\ge 1$ and a
positive sequence $(\rho_n)_n$ decreasing to zero so that for
every $m\ge m_0$ and $j=j(m)$ the intersection of the intervals
$f_{[0^m10^j]}(J)$ and $J$ is nonempty and that for every $x\in J$
we have
\[
    e^{-j\rho_j-m\rho_m} \le
    \lvert (f_{[0^m10^j]})'(x)\rvert
    \le e^{j\rho_j+m\rho_m}
   .
\]
Recall now the choice of the constants $\gamma$ and $\ga'$ in (F1.i).
Therefore, for every $m\ge m_0$ and $j=j(m)$, every $\ell\ge
\ell_0$, and every $x\in[0,1]$ we obtain
\begin{equation}\label{e.chat}
    {(\ga')}^\ell \,e^{-j\rho_j-m\rho_m}
    \le
    \lvert (f_{[1^\ell 0^m10^j]})'(x)\rvert
    \le e^{j\rho_j+m\rho_m}\gamma^\ell.
\end{equation}

Let us now choose $\ell=\ell(m,j)\ge \ell_0$ that is the smallest
number such that the right-hand side in~\eqref{e.chat} is $<1$,
this means that we have
\begin{equation}\label{e.musik}
    \frac{j\rho_j+m\rho_m}{-\log\gamma}
    < \ell
    \le \frac{j\rho_j+m\rho_m}{-\log\gamma}+1.
\end{equation}
Since $f^\ell_1([0,1])\subset J$  we can apply the above estimates
to any point $x\in f^\ell_1([0,1])$. Thus, the map
$f_{[1^\ell0^m10^j]}$ is a contraction in $[0,1]$ and hence has a
unique fixed point $p_{[1^\ell 0^m10^j]}$ whose basin of
contraction contains $[0,1]$. Moreover, its Lyapunov exponent
$\chi(p_{[1^\ell0^m10^j]},\xi)$ with $\xi=(1^\ell0^m10^j)^\ZZ$ satisfies
\[
 M(\ell,m,j)\eqdef    \frac{\ell\log \ga' -j\rho_j-m\rho_m}{j+m+\ell+1}
     \le \chi(p_{[1^\ell0^m10^j]},\xi)
     <0.
\]
Using~\eqref{e.musik}, we obtain that  $M(\ell,m,j)$ is bounded
from below by
\[
    \frac{1}{j+m+1}\log \ga'
        \left(\frac{j\rho_j+m\rho_m}{-\log\gamma}+1\right)
    -\frac{1}{j+m+1}\left(j\rho_j+m\rho_m\right)\le M(\ell,m,j).
\]
When we now take $m$ arbitrarily large the index $j=j(m)$ is also
large. Hence the  lower bound $M(\ell,m,j)$ is arbitrarily close
to $0$. As the exponent is negative, this finishes the proof of
the proposition.
\end{proof}

\begin{proof}[Proof of Proposition~\ref{p.posspec}]
    Take the fundamental domain $I=[f_0^{-1}(b_0),b_0]$.
    By Lemma~\ref{l.looping}, there exists a number $m_0\ge 1$ and a positive sequence $(\rho_k)_k$ decreasing to zero so that for every $m\ge m_0$ and $j=j(m)$ the intersection of the intervals $f_{[0^m10^j]}(I)$ and $I$ is nonempty and that for every $x\in I$ we have
\begin{equation}\label{e.almocofab}
    e^{-j\rho_j-m\rho_m}
   \le  \lvert (f_{[0^m10^j]})'(x)\rvert
   \le e^{j\rho_j+m\rho_m}.
\end{equation}
This implies, possibly after slightly decreasing $(\rho_k)_k$, that
\begin{equation}\label{e.feira}
    \big\lvert f_{[0^m10^j]}(I) \cap \big( f_0^{-1}(I)\cup I\big) \big\rvert
    \ge e^{-j\rho_j-m\rho_m}.
\end{equation}
We now consider the sub-interval
\[
    J_{m,j}\eqdef f_{[0^m10^j]}(I) \cap \big( f_0^{-1}(I)\cup I\big).
\]
and consider its expanded successors as defined in Definition~\ref{d.expanding}.

First recall that by Remark~\ref{r.import} the number of
applications of the maps $f_0$ and $f_1$ involved in the
definition of the expanded successor of an interval is uniformly
bounded from above and below by numbers $N_1>N_2\ge 1$ that do not
depend on the interval. Moreover, recall that by
Remark~\ref{r.import} each expanded itinerary has a uniform
expansion  bounded from below and above by numbers
$\kappa_1>\kappa_2>1$ that are independent of the itinerary.

Therefore, as the length of the interval in~\eqref{e.feira} is
bounded from below and each expanded successor involves a uniform
expansion bounded from below by $\kappa_2$, we need to repeat a
finite number $\overline N=\overline N(m,j)$  of times the
expanded successors to obtain the covering of the fundamental
domain $[f_0^{-2}(b_0),f_0^{-1}(b_0)]$ as stated in
Lemma~\ref{l.successor}. Now we denote by $F_{m,j}$ the resulting
concatenated map. Moreover, by construction the interval
\[
    f_{[\xi_0\ldots\xi_n]}(J_{m,j})\eqdef (f_0\circ F_{m,j})(J_{m,j}) ,
\]
covers the original interval $I$ and hence there exists an
expanding fixed point $p_{m,j}\in I$ with respect to the map
$f_{[\eta]}$ with $\eta\eqdef (0^m10^j\xi_0\ldots\xi_n)$. Moreover, by the comments above, $n$ is some number satisfying $\overline N N_2+1\le n<\overline
NN_1+1.$

We finally estimate the Lyapunov exponent of $p_{m,j}$. Using the
length estimate of $J_{m,j}$ in \eqref{e.feira} and the uniform
expansion of each expanded successor by a factor of at least
$\kappa_2$, we can estimate $\overline N$ from above by
\begin{equation}\label{e.N}
    \overline N\le  \frac{C+j\rho_j+m\rho_m}{\log\,\kappa_2},
\end{equation}
where $C>0$ only depends on the length $[f_0^{-2}(b_0),
f_0^{-1}(b_0)]$. Hence, by~\eqref{e.almocofab} and since each
expanding iterate expands at most by $\kappa_1$, the Lyapunov exponent
at $p_{m,j}$ satisfies
\[
    0<\chi(p_{m,j},\eta)
    \le \frac{\overline N \log\,\kappa_1+j\rho_j+m\rho_m}{\overline N \, N_2+j+m}\le  \frac{\overline
N \log\,\kappa_1+j\rho_j+m\rho_m}{j+m},
\]
Now~\eqref{e.N} implies that the upper bound can be estimated from
above by
\[
    \chi(p_{m,j},\eta)\le
    \frac{C+j\rho_j+m\rho_m}{\log\,\kappa_2}\frac{\log\,\kappa_1}{j+m}
    +\frac{j\rho_j+m\rho_m}{j+m}.
\]
Recall that the index $j=j(m)$ is large when $m$ is large.
Thus, $\rho_j$, $\rho_m\to 0$ and hence  this exponent is arbitrarily
close to $0$.
\end{proof}

%------------------------------------------------------------------------------------------
\subsection{Admissible domains}\label{ss.onedim.fix}
%------------------------------------------------------------------------------------------

In this subsection we explore the rich structure of admissible domains.

\begin{nota}{\rm
	Given $\xi=(\ldots\xi_{-1}.\xi_0\xi_1\ldots)\in\Sigma_2$ and $m\ge 1$, let us denote
\begin{equation}\label{def:I}
	I_{[\xi_{-m}\ldots\xi_{-1}.]}
    	\eqdef f_{\xi_{-1}}\circ\ldots\circ f_{\xi_{-m}}([0,1]).
\end{equation}
This set is always a non-trivial sub-interval of $[0,1]$. Note that $f_{\xi_i}([0,1])\subset[0,1]$ for every $\xi_i\in\{0,1\}$. Therefore, for each one-sided infinite sequence $\xi=(\ldots \xi_{-2}\xi_{-1}.)$ the sets $I_{[\xi_{-m}\ldots\xi_{-1}.]}$ form a nested sequence of non-empty compact intervals. Thus, the set $I_{[\xi]}$ defined by
\[
    	I_{[\xi]}\eqdef
    	\bigcap_{m\ge1} I_{[\xi_{-m}\ldots\xi_{-1}.]}
\]
is either a singleton or a non-trivial interval. For completeness, for each $n\ge0$
we write
\[
    	I_{[\xi_{-m}\ldots\xi_{-1}.\xi_0\ldots\xi_n]}
    	\eqdef I_{[\xi_{-m}\ldots\xi_{-1}.]}.
\]
}\end{nota}

\begin{rema}\label{r.Lor}{\rm
    Note that given a sequence $\xi\in\Sigma_2$, any point $x\in I_{[\xi]}$ is admissible for $\xi$. Observe that for all $m\ge 1$ the interval $I_{[\xi_{-m}\ldots\xi_{-1}.]}$ is the maximal domain of the map $f_{[\xi_{-m}\ldots \xi_{-1}.]}$ which justifies our notation.
}\end{rema}

Note that any sequence $\xi\in\Sigma_2$ is given by $\xi=\xi^-.\xi^+$, where $\xi^+\in\Sigma_2^+\eqdef\{0,1\}^\NN$ and $\xi^-\in\Sigma_2^-\eqdef\{0,1\}^{-\NN}$.

\begin{prop}\label{p.residualtrivfib}
We have the following properties:
\begin{enumerate}
\item [1)] The set $\{\xi\in\Sigma_2\colon I_{[\xi]} \text{ is a single point}\,\}$ is
residual in $\Sigma_2$.
\item [2)] Given $\xi^+\in\Sigma_2^+$, the set $\{\xi^-\in \Si_2^-\colon I_{[\xi^-.\xi^+]} \text{ is non-trivial}\,\}$ is uncountable and dense in $\Sigma_2^-$.
\end{enumerate}
Moreover, for every closed interval $J\subset[0,1)$ and every $\xi^+\in\Sigma_2^+$ the set of sequences $\xi^-\in\Sigma_2^-$ with $I_{[\xi^-.\xi^+]}\supset J$ is uncountable.
\end{prop}

We postpone the proof of this proposition to the end of this section. Let us first collect some basic properties of the admissible domains $I_{[\xi]}$.

\begin{lemm}\label{l.in01}
    Given a finite sequence $(\xi_{-m}\ldots\xi_{-1})$, if $\xi_{-k}=1$ for at least two indices $k\in\{1,\ldots,m\}$ then $I_{[\xi_{-m}\ldots\xi_{-1}.]}\subset(0,1)$.
\end{lemm}

\begin{proof}
    Let $m\ge1$ be the smallest index so that $\xi_{-m}=1$. By property {(F01)} we have $I_{[\xi_{-m}\ldots\xi_{-1}.]}([0,1])=[0,c_1]\subset[0,1)$.
    Recall that $f_0([0,1])=[0,1]$ and $f_0(0)=0$. If $k\ge m$ is the smallest index such that $\xi_{-k}=1$ then we have
    \[
        I_{[\xi_{-k}\ldots\xi_{-m}\ldots\xi_{-1}.]}
        =\big[f_1\circ f_0^{k-m-1}(c_1),f_1(0)\big],
    \]
    where $f_1\circ f_0^{k-m-1}(c_1)>0$ and $f_1(0)<1$.
    This implies that for all $\ell\ge k$ we have $I_{[\xi_{-\ell}\ldots\xi_{-1}.]}\subset(0,1)$. This proves the lemma.
\end{proof}

Note that $f_0^k([0,1])=[0,1]$ implies that
\[
	I_{[0^k\xi_{-m}\ldots\xi_{-1}.]}
	= ( f_{\xi_{-1}}\circ \ldots \circ f_{\xi_{-m}})
		\big(f_0^k([0,1])\big)
	= I_{[\xi_{-m}\ldots\xi_{-1}.]}.
\]
This implies the following result.

\begin{lemm}\label{l.leftint}
    For all $k\ge 1$ we have $I_{[0^k\xi_{-m}\ldots\xi_{-1}.]} = I_{[\xi_{-m}\ldots\xi_{-1}.]}$.
\end{lemm}

\begin{lemm}\label{l.nontri}
    Given $\xi=(\ldots\xi_{-1}.\xi_0\xi_1\ldots)\in\Sigma_2$ satisfying
    $\xi_{-m}=0$ for all $m\ge m_0$ for some index $m_0\ge 1$, the domain
    $I_{[\xi]}=I_{[\xi_{-m_0}\ldots\xi_{-1}.]}$ is a non-trivial interval.
\end{lemm}

\begin{proof}
    Recall that $I_{[\xi_{-m}\ldots\xi_{-1}.]}$ for any $m\ge 1$ is a non-trivial interval. Hence, the claim follows from Lemma~\ref{l.leftint}.
\end{proof}

We now start investigating  the structure of admissible domains. Recall the choice of constants $\gamma$ and $\beta$ in  {(F1.i)} and {(F0.i)}.

\begin{defi}\label{d.defcontr}
{\rm
    We call a sequence $\xi=(\ldots\xi_{-1}.\xi_0\xi_1\ldots)\in \Sigma_2$ \emph{asymptotically contracting} if for every $m\ge 1$ we have
    that
    \begin{equation}\label{e.defcontr}
    \limsup_{m\to\infty}\gamma^{k_m}\, \be^{m-k_m}=0,\quad
    \text{ where }k_m=\sum_{i=1}^{m} \xi_{-i} .
    \end{equation}
}\end{defi}

With~\eqref{def:I} and properties {(F0.i)} and {(F1.i)}, if $\xi$ is asymptotically contracting we get that the length of the interval $I_{[\xi]}$ satisfies
\[
    \lvert I_{[\xi]}\rvert \le \lim_{m\to\infty}\gamma^{k_m}\beta^{m-k_m} =0.
\]
The following lemma is hence an immediate consequence.

\begin{lemm}\label{l.lp}
    For every asymptotically contracting sequence $\xi\in \Sigma_2$  the set $I_{[\xi]}$ consists of a single point.
    Moreover, if $\xi=(\xi_0\ldots\xi_{m-1})^\ZZ$ is periodic, then this point is an  attracting fixed point of the map $f_{[\xi_0\ldots\xi_{m-1}]}$.
\end{lemm}

\begin{rema}\label{r.ascon}
{\rm Given any asymptotically contracting sequence
$\eta=\eta^-.\eta^+\in\Sigma_2$, the set $\{\eta^-.\xi^+\colon
\xi^+\in\Sigma_2^+\}$ consists of only asymptotically contracting
sequences. Observe that this set clearly is uncountable.
}\end{rema}

\begin{lemm}\label{l.upsemi}
    The function $\xi\mapsto \lvert I_{[\xi]}\rvert$ is upper semi-continuous but not continuous. However, it is continuous at every $\eta\in\Sigma_2$ for that $I_{[\eta]}$ is a single point.
\end{lemm}

\begin{proof}
    Let $\eta\in \Sigma_2$.
    Since $I_{[\eta_{-n}\ldots\eta_{-1}.]}$ form a nested sequence of compact intervals, their length is non-increasing and for any $\varepsilon>0$ there exists $N\ge 1$ such that $\lvert I_{[\eta]}\rvert  \le \lvert I_{[\eta_{-n}\ldots\eta_{-1}.]}\rvert  < \lvert I_{[\eta]}\rvert + \varepsilon$ for every $n\ge N$. For every $\xi\in\Sigma_2$ with $d(\xi,\eta)\le\delta\eqdef\sum_{\lvert i\rvert>N}2^{-\lvert i\rvert}$ we have $\xi_i=\eta_i$ for every $\lvert i\rvert\le n$ and thus $\lvert I_{[\xi_{-n}\ldots\xi_{-1}.]}\rvert=\lvert I_{[\eta_{-n}\ldots\eta_{-1}.]}\rvert$.
    Since $I_{[\xi_{-n}\ldots\xi_{-1}.]}$ is also nested, we obtain
    \[
        \lvert I_{[\xi]}\rvert \le \lvert I_{[\xi_{-n}\ldots\xi_{-1}.]}\rvert
    = \lvert I_{[\eta_{-n}\ldots\eta_{-1}.]}\rvert
    < \lvert I_{[\eta]}\rvert + \varepsilon,
    \]
which implies upper semi-continuity.

If $I_{[\eta]}$ is a single point only then $\lvert
I_{[\eta_{-n}\ldots\eta_{-1}.]}\rvert= \lvert
I_{[\xi_{-n}\ldots\xi_{-1}.]}\rvert<\varepsilon$ and
    \[
    \big\lvert\lvert I_{[\eta]}\rvert - \lvert I_{[\xi]}\rvert\big\rvert
    \le
    \big\lvert \lvert I_{[\eta]}\rvert - \lvert I_{[\eta_{-n}\ldots\eta_{-1}.]}\rvert \big\rvert
    +  \lvert I_{[\xi_{-n}\ldots\xi_{-1}.]}\rvert - \lvert I_{[\xi]}\rvert
        \\
    \le 2\varepsilon.
    \]
    This    implies continuity at $\eta$.

    Finally, observe that the function $\xi\mapsto \lvert I_{[\xi]}\rvert$ is not continuous in general. Indeed, taking $\eta=(0^\ZZ)$ recall that $I_{[\eta]}=[0,1]$. However, as every sequence $\xi\in\Sigma_2$ for that $\xi_{-i}=1$ for all $i$ large enough is asymptotically contracting and hence, by Lemma~\ref{l.lp},   the domain $I_{[\xi]}$ contains only a point. Clearly, such $\xi$ can be chosen arbitrarily close to $\eta$.
\end{proof}

To complete our analysis of admissible domains  let us consider the sets $I_{[\xi]}$ that contain a repelling point.

\begin{lemm}\label{l.cont}
    For every periodic sequence $\xi=(\xi_0\xi_1\ldots\xi_{m-1})^\ZZ\in \Sigma_2\setminus \{0^\ZZ\}$ for that the map $f_{[\xi_0\ldots \xi_{m-1}]}$ possesses a repelling fixed point $q$, the set $I_{[\xi]}$ is a non-trivial interval. Moreover, this interval contains points $p_\infty$ and $\widetilde p_\infty$ with $p_\infty<q<\widetilde p_\infty$ that are fixed with respect to $f^2_{[\xi_0\ldots \xi_{m-1}]}$.
\end{lemm}

\begin{proof}
    Assume that $q$ is a fixed point with respect to $f_{[\xi_0\ldots\xi_{m-1}]}$ that is repelling.
    Consider a point $p\in W^\ut_\loc( q,f_{[\xi_0\ldots\xi_{m-1}]})$ and
    note that the sequence of points $p_k\eqdef f_{[\xi_0\ldots\xi_{m-1}]}^{2k}(p)$, $k\ge 1$, is well-defined. Let us assume that $p<q$.
    Note that $f_{[\xi_0\ldots\xi_{m-1}]}^{2k}$ preserves orientation and hence
    we can consider the limit $p_\infty\eqdef\lim_{k\to\infty}p_k\in [0,1]$.
    Observe that $f_{[\xi_0\ldots\xi_{m-1}]}^2(p_\infty)=p_\infty$.
    Since $p\in W^u_\loc(q,f_{[\xi_0\ldots\xi_{m-1}]})$ and $p\ne q$, we conclude $p_\infty\ne q$.
    Then, since $f^2_{[\xi_0\ldots \xi_{m-1}]}$ preserves orientation, we can conclude that the interval $[p_\infty,q]$ is contained in $I_{[\xi]}$.

    Completely analogous, we can show that $I_{[\xi]}$ contains an interval $[q,\widetilde p_\infty]$ where $\widetilde p_\infty= f^2_{[\xi_0\ldots \xi_{m-1}]}(\widetilde p_\infty)\ne q$.
    This completes the proof of the lemma.
\end{proof}

Observe that, as a consequence of Proposition~\ref{p.posspec}, there are infinitely many periodic sequences $\xi$ such that $I_{[\xi]}$ contains a repelling periodic point and therefore is non-trivial, by Lemma~\ref{l.cont}.

We now start by analyzing further properties of the admissible
domains. Continuing Remark~\ref{r.ascon}, we show that the set of
asymptotically contracting sequences is in fact much richer.

\begin{lemm}\label{l.uncounta}
    Given any $\xi^+\in\Sigma_2^+$, the set
    \[
    \cA_{\xi^+}\eqdef
    \big\{\eta=\eta^-.\xi^+\colon \eta^-\in\Sigma_2^-\text{ and }
        \eta\text{ is asymptotically contracting} \big\}
    \] is uncountable.
\end{lemm}

\begin{proof}
    Fix any $\xi^+\in\Sigma_2^+$. To prove that the set $\cA_{\xi^+}\subset\Sigma_2$   is uncountable we use the standard Cantor diagonal argument.

    Arguing by contradiction, we assume that $\cA_{\xi^+}$ is countable.
    Notice that for every asymptotically contracting sequence $\xi=(\ldots\xi_{-1}.\xi_0\xi_1\ldots)\in\Sigma_2$,
    we cannot have $\xi_{-i}=0$ eventually for all large enough $i$.
    Let us consider only the subset $\widetilde\cA_{\xi^+}\subset\cA_{\xi^+}$ of
    sequences $\xi=(\ldots\xi_{-1}.\xi^+)$ for that $\xi_{-i}=0$ for infinitely many $i$.  Clearly, this set is also countable. Consider some numeration of it
    \[
    \widetilde\cA_{\xi^+}
        =\left\{\xi^1=(\ldots1^{L_2^1}0^{K_2^1}1^{L_1^1}0^{K_1^1}.\xi^+),
        \xi^2=(\ldots1^{L_2^2}0^{K_2^2}1^{L_1^2}0^{K_1^2}.\xi^+),
        \ldots\right\}.
    \]
    Here we allow also $K^k_1=0$ in which case the symbol $0$ is neglected.
    We now construct a ``new" sequence of $\widetilde\cA_{\xi^+}$ that is not in that numeration.
    Let $\overline L_1=L_1^1+1$ and for $k\ge 2$ choose some number
    \[
    \overline L_{k+1}> \max\big\{\overline L_k, L_{k+1}^{k+1}\big\}
    \quad\text{ such that }\quad
    \beta\cdot\gamma^{\sum_{i=1}^{k+1}\overline L_i}\cdot\beta^k<\frac{1}{2^k}.
    \]
    Observe that the sequence $(\ldots01^{\overline L_2}01^{\overline L_1}0.\xi^+)$ is asymptotically contracting.
    By construction this sequence in not in the enumeration of $\widetilde\cA_{\xi^+}$ above, that is a contradiction. Hence, $\widetilde \cA_{\xi^+}$ (and thus
   $\cA_{\xi^+}$)
   is uncountable.
\end{proof}

The following lemma shows that the set of sequences $\xi$ having a non-trivial admissible domain $I_{[\xi]}$ is also very rich.

\begin{lemm}\label{l.uncountb}
    For every closed interval $J\subset(0,1)$ and every $\xi^+\in\Sigma_2^+$ the set
   \[
    \cA_{\xi^+,J}\eqdef
    \big\{\xi=(\xi^-.\xi^+)\colon
    \xi^-\in\Sigma_2^-\text{ and }
    I_{[\xi]}\supset J
    \big\}
    \]
    is uncountable.
\end{lemm}

\begin{proof}
    First, let us show that $\cA_{\xi^+,J}$ is nonempty.
    Note that given any $J\subset(0,1)$, there exists a number $k=k(J)\ge 1$
    such that $f_0^{-k}(J)\subset (0,c_1)$, where $c_1= f_1(0)$. Let now  $J_0\eqdef J$ and $K_1\eqdef k(J_0)=k(J)$ and note that by definition in equation~\eqref{def:I} we have
    \[
    J_0\subset  f_0^{K_1}([0,c_1]) =(f_0^{K_1}\circ f_1)([0,1]) =I_{[10^{K_1}.]} .
    \]
    For $\ell\ge 0$ let us now define recursively
    \[
    	K_{\ell+1}\ge
	k(J_\ell)\quad\text{ such that }\quad
	f_0^{-K_{\ell+1}}(J_\ell)\subset (0,c_1)
    \]
    and
    \[
    J_{\ell+1} \eqdef (f_1^{-1}\circ f_0^{-K_{\ell+1}})(J_\ell).
    \]
   With such a choice, we have
    \[
        J_{\ell+1}
    = (f_1^{-1}\circ f_0^{-K_{\ell+1}}\circ\ldots \circ f_1^{-1}\circ f_0^{-K_1} )(J)
    = f_{[10^{K_{\ell+1}}\ldots 10^{K_1}.]}(J)
    \subset [0,1]
    \]
    and hence $J\subset I_{[10^{K_{\ell+1}} \ldots 10^{K_1}.]}$ for every $\ell\ge 1$  (recall Remark~\ref{r.Lor}).
    This implies $J\subset I_{[\xi]}$ for the sequence $\xi=(\ldots 10^{K_\ell}\ldots10^{K_1}.\xi^+)$. This proves that $\cA_{\xi^+,J}$ is nonempty.

	We point out that, since we have $f_0^{-m}(J_\ell)\subset (0,c_1)$ for every $m> K_{\ell+1}$, we can repeat the construction above replacing in each step $K_\ell$ by any number $\overline{K}_\ell> K_\ell$. In this way, we get a new sequence $\overline \xi$ such that $J\subset I_{[\overline \xi]}$. In particular, this implies that the set of sequences $\xi$ such that $J\subset I_{[\xi]}$ is infinite.

 The above remark guarantees that the set $\cA_{\xi^+,J}$ is infinite. To prove that $\cA_{\xi^+,J}$ is uncountable we use again the Cantor diagonal argument.  Arguing by contradiction, we assume that $\cA_{\xi^+,J}$ is countable. Let us consider the subset $\widetilde\cA_{\xi^+,J}\subset\cA_{\xi^+,J}$ defined by
    \[\begin{split}
        \widetilde\cA_{\xi^+,J}\eqdef
    \Big\{\xi&=(\ldots 0^{K_3}10^{K_2}10^{K_1}.\xi^+)
    \in\Sigma_2\colon\\
    &\text{ there are infinitely many blocks of $0$s of length }K_\ell\\
    &\text{ satisfying }
    K_{\ell+1}>K_\ell\text{ for all }\ell\ge 1 \\
    &\text{ and }
    I_{[\xi_{-m}\ldots\xi_{-1}.]}\supset J\text{ for all }m\ge 1\Big\}
    \end{split}\]
    and consider its numeration
    \[
    \widetilde\cA_{\xi^+,J}=\left\{
            \xi^1=(\ldots10^{K_\ell^1}\ldots10^{K_1^1}.\xi^+),
            \xi^2=(\ldots10^{K_\ell^2}\ldots10^{K_1^2}.\xi^+),
        \ldots\right\}.
    \]
Let now $J_0'=J$ and  $\overline K_1= \max\{K_1^1+1, k(J_0')+1\}$ with
    $J_0'\subset f_{[0^{\overline K_1}]} ([0,c_1])$ and write
    $J_1'\eqdef (f_1^{-1} \circ f_0^{-\overline K_1})([0,1])$. Note that
    $J_1'\subset I_{[10^{\overline K_1}]}$.
   Arguing inductively,
    for $\ell\ge 2$ let us choose a number $\overline K_{\ell+1}$ such that
    \[
    \overline K_{\ell+1}> \max\big\{\overline K_\ell, K_{\ell+1}^{\ell+1}, k(J_\ell')+1\big\}
    \quad\text{and}\quad
J_{\ell+1}' \eqdef (f_1^{-1}\circ f_0^{- \overline{K}_{\ell+1}})(J_\ell').
   \]
Bearing in mind the above remark and arguing as
above, these choices give
$$
    I_{[10^{\overline K_{\ell+1}}10^{\overline K_\ell}1\ldots10^{\overline K_1}.]}\supset J_{\ell+1}'.
$$
    Clearly, none of the  sequence $(\ldots10^{\overline K_\ell}\ldots10^{\overline K_1}.\xi^+)$ is in $\widetilde\cA_{\xi^+,J}$,
    contradicting that $\widetilde\cA_{\xi^+,J}$ is countable. Hence $\cA_{\xi^+,J}$ is uncountable.
\end{proof}

We finally provide the proof of our proposition.

\begin{proof}[Proof of Proposition~\ref{p.residualtrivfib}]
    We first prove that the set of sequences with trivial spines is residual. As an immediate consequence of Definition~\ref{d.defcontr}, given any sequence $\xi=(\ldots\xi_{-1}.\xi_0\xi_1\ldots)\in\Sigma_2$, for any $m\ge 1$  the sequence  $\xi'=(\ldots11\xi_{-m}\ldots\xi_{-1}.\xi_0\xi_1\xi_2\ldots)$ is asymptotically contracting. Moreover, by Lemma~\ref{l.lp} the domain $I_{[\xi']}$ is a single point only. Clearly, the distance between $\xi$ and $\xi'$ can be made arbitrarily small when increasing $m$. This proves that the sequences $\xi$ such that $I_{[\xi]}$ is trivial are dense in $\Sigma_2$.
    Given $\varepsilon>0$, consider the set
    \[
    \cA_\varepsilon\eqdef
    \big\{\xi\in\Sigma_2\colon \lvert I_{[\xi]}\rvert\le\varepsilon\big\}.
    \]
    The second statement in Lemma~\ref{l.upsemi} in particular implies that $\cA_\varepsilon$ contains an open and dense subset of $\Sigma_2$. Thus, the set $\bigcap_{n\ge1}\cA_{1/n}$ contains a residual subset that consists of sequences for that $I_{[\xi]}$ is a single point. This proves the first part of the proposition.

	We now look at the set of sequences with non-trivial spines.
    Given any sequence $\xi=(\ldots\xi_{-1}.\xi_0\xi_1\ldots)\in\Sigma_2$, recall that $I_{[\xi_{-m}\ldots\xi_{-1}.]}$ for any $m\ge 1$ is
    a non-trivial interval. By Lemma~\ref{l.leftint} we have $I_{[0^k\xi_{-m}\ldots\xi_{-1}.]} = I_{[\xi_{-m}\ldots\xi_{-1}.]}$ for any $k\ge 1$.
    Further, recall that $I_{[\xi_{-m}\ldots\xi_{-1}.]}=I_{[\xi_{-m}\ldots\xi_{-1}.\xi_0\ldots\xi_m]}$. Thus, the sequence  $\xi'=(\ldots00\xi_{-m}\ldots\xi_{-1}.\xi_0\xi_1\ldots)$ satisfies $I_{[\xi']}= I_{[\xi_{-m}\ldots\xi_{-1}.]}$ and hence contains an interval. Clearly, the distance between $\xi$ and $\xi'$ can be made arbitrarily small when increasing $m$. Together with Lemma~\ref {l.uncountb}, this proves the second part of the proposition.
\end{proof}

%-------------------------------------------------------------------------------------------------------
\subsection{Gap in the Lyapunov spectrum}
%-------------------------------------------------------------------------------------------------------

We finally establish some gap in the spectrum of Lyapunov exponents.

\begin{prop}[Spectral gap]\label{p.beta}
    Let
    \[
                \widetilde\chi\eqdef
                \sup\big\{\chi(p,\xi)\colon
                p\in[0,1],\, \xi\in\Sigma_2\setminus E\big\},
    \]
    where
	\[
    		E\eqdef
		\{\xi=(\xi^-.\xi^+)\colon \xi^+=(\xi_0\ldots\xi_k 0^\NN),k\ge 0\}.
	\]	
    Then $\widetilde\beta\eqdef\exp\widetilde\chi<\beta$.
\end{prop}

\begin{proof}
The idea of the proof is quite simple although the proof itself is a bit technical. Note that the exponent $\chi(p,\xi)$ could be close to $\beta$ only if the orbit $\{f_{[\xi_0\ldots\,\xi_m]}(p)\}_{m\ge1}$ stays arbitrarily close to $0$ infinitely often. Note also that each visit close to $0$ was preceded by a visit close to $1$. This implies that the effect of expansion (iterates of $f_0$ close to $0$) will be compensated by a (previous) contraction (iterates of $f_0$ close to $1$ and some iterates of $f_1$) that will force the exponent to decrease. Now we will provide the details.

For simplicity of the exposition we assume that $f_0$ is non-linear in a neighborhood of $0$. A similar proof can be done in the general case.
Consider a number $\de>0$ close to $0$ and define the sets
$$
H_0\eqdef [0,\delta],\quad
H'_0\eqdef f_1^{-1} (H_0) = [h_0',1].
$$
Recalling condition (F01), we have that if $\delta$ is small enough then
\begin{equation}\label{e.H0}
	f_1(H_0)\cap H_0 =\emptyset,
	\quad f_1(H'_0)\cap H'_0 =\emptyset,
	\quad f_1([0,1])\cap H_0' =\emptyset.
\end{equation}

We first introduce some constants that will be used throughout the proof:
\begin{equation}\label{e.ines}
	\widetilde\alpha\eqdef \min_{H_0'}\,\lvert f_1'\rvert,
	\quad\widehat\alpha\eqdef \max_{H_0'}\,\lvert f_1'\rvert.
\end{equation}
Note that $\widetilde\alpha/\widehat\alpha\sim 1$ if $\delta$ is small enough.
Let us further define
\[
	\beta_0\eqdef \sup\{f_0'(x)\colon x\notin H_0\}
		= \sup\{f_0'(x),f_1'(x)\colon x\notin H_0\}<\beta.
\]
Observe that $\beta_0<\beta$ follows from our simplifying assumption that $f_0$ is non-linear close to $0$. Let
$$
	\be_0' \eqdef \inf\{f_0'(x) \, \colon \, x\in H_0\}<\beta,\quad
	\la'_0 \eqdef \sup\{f_0'(x) \, \colon \, x\in H'_0\}.
$$
Note that  if $\de$ is small then $\la_0'$ and $\be_0'$ are close to $\la$ and $\be$ and thus
\begin{equation}\label{e.derivatives}
	 \lvert\log \la\rvert\, \frac{\log \beta}{\log\be_0'}
	- \lvert\log \la_0'\rvert  +\frac{\log \be}{2}< \log \be_0'.
\end{equation}

 To prove the proposition, note that it is enough to consider the case that $\xi_i=0$ for infinitely many $i\ge 1$. Indeed, otherwise, because $f_1$ is a contraction, we have $\chi(p,\xi)<0$. Moreover, by replacing $p$ by some iterate, we can assume that $p\ne 0$. Further, we can assume that the orbit $\{f_{[\xi_0\ldots\,\xi_m]}(p)\}_{m\ge 0}$ hits the interval $H_0$ infinitely many times. Indeed, otherwise this orbit is contained in the interval $(\de,1]$ in which the derivatives $f_0'$ and $f_1'$ are upper bounded by $\be_0'$ and thus the Lyapunov exponent of $\chi(p,\xi)$ is upper bounded by $\log \be_0'$. Hence, without loss of generality, possibly replacing $p$ by some positive iterate, we can assume that $p \in H_0$ and $f_{\xi_0}(p)\notin H_0$.

For every $m\ge 0$ let us write
\[
	p_{m+1}\eqdef f_{[\xi_0\ldots\,\xi_m]}(p).
\]	
We define three increasing sequences $(r_k)_k$, $(e_k)_k$, $(i_k)_k$ of positive integers as follows (compare Fig.~\ref{f.seiiq}): $i_k< r_k \le e_k<i_{k+1}$,
\begin{equation}\label{e.rkek}
	p_j\in H_0
	\quad\text{ if and only if }\quad
	r_k\le j \le e_k \text{ for some }k,
\end{equation}
and
\begin{equation}\label{e.ik}
	p_j\in H_0'\quad\text{ if and only if }\quad
	i_k\le j\le r_k-1 \text{ for some }k.
\end{equation}
\vspace{0.5cm}
\begin{figure}[h]
\begin{minipage}[c]{\linewidth}
\centering
\begin{overpic}[scale=.5]{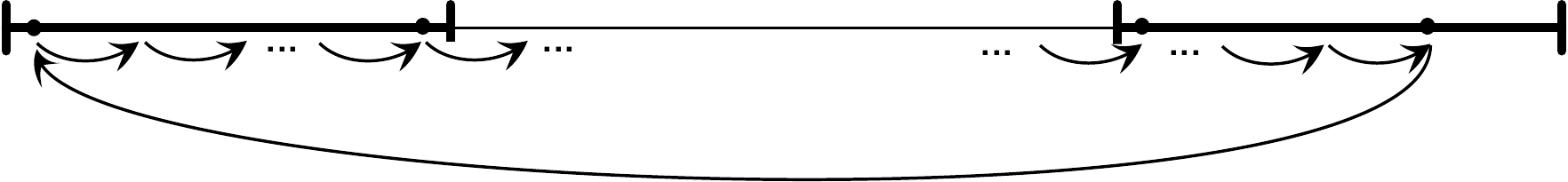}
    \put(-0.5,4){$0$}	
    \put(99,4){$1$}
    \put(1.4,14){$p_{r_k}$}
    \put(26,14){$p_{e_k}$}
    \put(71,14){$p_{i_k}$}
    \put(90,14){$p_{r_k-1}$}
    \put(13,17){$H_0$}
    \put(82,17){$H_0'$}
  \end{overpic}
\end{minipage}
\caption{Definition of the sequences $(r_k)_k$, $(e_k)_k$, $(i_k)_k$}
\label{f.seiiq}
\end{figure}
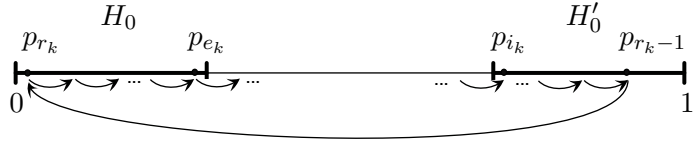

Note that indeed by our choices the only way of entering $H_0$ is by coming from $H_0'$ after applying $f_1$ and the only way from entering in $H_0'$ is after applying $f_0$. More precisely:
	Since $f_1(H_0)\cap H_0=\emptyset$, recall \eqref{e.H0}, we have that
$\xi_j=0$ for every index $j\in \{r_k,\ldots,e_k-1\}$ whenever $r_k<e_k$.
By definition of $r_k$ we have $p_{r_k}\in H_0$ and $p_{r_k-1}\not
\in H_0$ and thus $p_{r_k-1}>p_{r_k}$. Since $f_0$ is an increasing function, we have that $p_{r_k}\ne f_0(p_{r_k-1})$. Thus,
\[
	p_{r_k}=f_1(p_{r_k-1}),\quad
	p_{r_k-1}\in H_0',\quad \text{ and }\quad
	\xi_{r_k-1}=1.
\]	
Since $f_1([0,1])\cap H_0'=\emptyset$, recall \eqref{e.H0},  we have $\xi_j=0$ for every index $j\in \{i_k-1,\dots,r_k-2\}$. In particular this implies that
\begin{equation}\label{e.zeros}
	p_{i_k} \in [h_0',f_0(h_0')).
\end{equation}

By the definitions of $H_0$ and of the sequences above we have
that $p_j\notin H_0$ for all $j\in \{e_k+1, i_{k+1}-1\}$, and therefore
\begin{equation}
\label{e.estimate1}
 	\log\big\lvert
	\big( f_{[\xi_{e_k+1}\,\ldots \,\xi_{i_{k+1}-1}]} \big)'(p_{e_{k}+1})\big\rvert
	< (\beta_0)^{i_{k+1}-e_k -1}.
\end{equation}

Let us denote by $N_k$ the number of iterates of the point $p_{i_k}$ in $H_0'$, that is
\[
	N_k\eqdef r_k-i_k-1.
\]	

\begin{claim}\label{c.alphaa}
	We have
	$\displaystyle p_{r_k}\ge
	\widetilde\alpha \,\widehat\alpha^{-1}\lambda^{N_k+1}\,\delta$.
\end{claim}

\begin{proof}
	Recall that by~\eqref{e.zeros} we have $p_{i_k}< f_0(h_0')$ and hence
	\[
		f_0^{N_k}(p_{i_k}) = p_{r_k-1}<
		f_0^{N_k}(f_0(h_0')) = f_0^{N_k+1}(h_0').
	\]	
	Since $h_0'=f_1^{-1}(\delta)$, with~\eqref{e.ines} we can estimate $1- h_0'\ge \widehat\alpha^{-1}\delta$. Hence, by (F0.i) we can estimate
	\[
		1- p_{r_k-1}
		> 1 -  f_0^{N_k+1}(h_0')
		\ge
		\lambda^{N_k+1}\,\widehat\alpha^{-1}\delta\,.
	\]
	Finally, since $f_1(0)=1$, we have
	\[
		p_{r_k}=f_1(p_{r_k-1})\ge
		\widetilde\alpha\,\lambda^{N_k+1}\,\widehat\alpha^{-1}\delta,
	\]
	which proves the claim.
\end{proof}

By Claim~\ref{c.alphaa}, $e_k-r_k$ is bounded from above by $\widetilde M_k+1$, where $\widetilde M_k$ is defined by
\[
	(\beta_0')^{\widetilde M_k} \,\widetilde\alpha\,\widehat\alpha^{-1} \la^{N_k+1} \, \de
	= \de,
\]	
that is,
\begin{equation}\label{e.utt}
	\widetilde M_k\le M_k\eqdef
	\left\lfloor
	\frac{(N_k+1)\, \lvert \log \la\rvert - \log\widetilde\alpha/\widehat\alpha}
		{\log \be_0'}\right\rfloor+1.
\end{equation}

Let us now estimate the finite-time Lyapunov exponents associated to each of the finite sequences $(\xi_{i_k}\dots \xi_{e_k})$.

\begin{claim}\label{e.2loop}
There exists $\beta_0''<\beta$ such that	
\[
	\frac{\log\big\lvert \big( f_{[\xi_{i_k}\ldots \,\xi_{e_k}]}\big)'(p_{i_k})
 	\big\rvert}{e_k-i_k+1} \le \log\beta_0''.
\]
\end{claim}

\begin{proof}
We can freely assume that the number of iterations in the interval $H_0$ is the maximum possible (clearly this is the
case that maximizes the derivative $f_{[\xi_{i_k}\dots\, \xi_{e_k}]}$). That is, let us suppose that $e_k-r_k=M_k$. Recalling the definition of $\widehat\alpha$ in~\eqref{e.ines}, with the above we obtain
\[
\frac{\log\big\lvert \big( f_{[\xi_{i_k}\ldots\, \xi_{e_k}]}\big)'(p_{i_k})
	\big\rvert}{e_k-i_k+1}
 \le \frac{M_k\log\be + \log\widehat\alpha - N_k\,\lvert \log \la_0'\rvert}
		{M_k+N_k+1}.
\]
From~\eqref{e.utt} we conclude
\[\begin{split}		
&\frac{\log\big\lvert \big( f_{[\xi_{i_k}\ldots\, \xi_{e_k}]}\big)'(p_{i_k})
	\big\rvert}{e_k-i_k+1}\\
&\le  \frac{1}{M_k+N_k+1}
	\left[
	\left(
		(N_k + 1) \, \lvert\log \la \rvert\,
		- \log\frac{\widetilde\alpha}{\widehat\alpha}
	\right)\frac{\log \be}{\log \be_0'}
			+ \log\beta
 			- N_k\,\lvert\log\lambda_0'\rvert
	\right]\\
&\le
	\frac{N_k}{M_k+N_k+1}\left(
		\lvert\log \la \rvert\,\frac{\log \be}{\log \be_0'}
		- \lvert\log\lambda_0'\rvert
		\right)
	+ \frac{\log\beta}{M_k+N_k+1}	\\
&\phantom{\le}		
	+ \frac{1}{M_k+N_k+1}
		\left( \lvert\log \la \rvert
 		- \log\frac{\widetilde\alpha}{\widehat\alpha}
	   \right)\frac{\log \be}{\log \be_0'}   \\
&\le
	\left(
		\lvert\log \la \rvert\,\frac{\log \be}{\log \be_0'}
		- \lvert\log\lambda_0'\rvert
		+ \frac{\log\beta}{2}
		\right)\\
&\phantom{\le}		
	+ \frac{1}{M_k+N_k+1}
	\left[
		\lvert\log \la \rvert - \log\frac{\widetilde\alpha}{\widehat\alpha}
	\right]  \frac{\log \be}{\log \be_0'} ,
\end{split}\]
where in the last line we also used $M_k+N_k+1\ge2$.
By~\eqref{e.derivatives}
\[
\frac{\log\big\lvert \big( f_{[\xi_{i_k}\ldots\, \xi_{e_k}]}\big)'(p_{i_k})
	\big\rvert}{e_k-i_k+1}
< \log \be_0'
+\frac{1}{M_k+N_k+1}
	\left[
		\lvert\log \la \rvert - \log\frac{\widetilde\alpha}{\widehat\alpha}
	\right]  \frac{\log \be}{\log \be_0'} .
\]

Let $L\ge1$ be large enough such that the right most term in the last estimate is less than $\varepsilon\eqdef(\log\beta-\log\beta_0')/2$ whenever $M_k+N_k+1\ge L$ and thus the claim is proved if $M_k+N_k+1\ge L$. In the finitely many remaining possible cases with $\lvert f_1'\rvert\le\gamma$ (recall (F1.i)) we can estimate that
\[
	\frac{\log\big\lvert \big( f_{[\xi_{i_k}\ldots\, \xi_{e_k}]}\big)'(p_{i_k})
	\big\rvert}{e_k-i_k+1}
	\le  \max_{\ell=1,\ldots,L}\frac{\log\,(\beta^\ell\,\gamma)}{\ell+1}
	<  \max_{\ell=1,\ldots,L}\frac{\ell}{\ell+1}\log\,\beta
	< \log\beta.
\]
Hence, with
\[
	\beta_0''\eqdef
	\exp\max\{\log\beta_0'+\varepsilon,\max_\ell \frac{\ell}{\ell+1}\log\beta\,\}
\]	
we have $\beta_0''<\beta$. This proves the claim.
\end{proof}

We are now ready to get an upper bound for $\chi(p,\xi)$. It is
enough to consider segments of orbits corresponding to exit times
$e_k$ and starting at the point $p_1$:
\[
\begin{split}
\frac{ \log\big\lvert \big( f_{[\xi_{1}\ldots \,\xi_{e_k}]} \big)'(p_{1})
	\big\rvert}{e_k}
\le&  \sum_{j=0}^{k}  \left( \frac{e_j-i_j+1}{e_k} \right) \,
\frac{ \log\big\lvert \big( f_{[\xi_{i_j}\ldots\, \xi_{e_j}]}\big)'(p_{i_j})
	\big\rvert}{e_j-i_j+1}  \\
& + \left( \frac{i_{j+1}-e_j-1}{e_k} \right)\,
 \frac{
\log\big\lvert \big( f_{[\xi_{e_j+1}\ldots\, \xi_{i_{j+1}-1}]}\big)'(p_{e_{j+1}})
	\big\rvert}{i_{j+1}-e_j-1}.
\end{split}
\]
By equations \eqref{e.estimate1} and Claim~\ref{e.2loop} we get that
this derivative is bounded from above by $\max\{\log \be_0,\log\beta_0''\}<\log\beta$. This completes the proof of the proposition.
\end{proof}

%------------------------------------------------------------------------------------------
\section{Transverse homoclinic intersections}\label{s.tranhomint}
%------------------------------------------------------------------------------------------

%------------------------------------------------------------------------------------------
\subsection{The maximal invariant set}
%------------------------------------------------------------------------------------------

In this section we are going to prove that the maximal invariant
set of $F$ in the cube $\CC$  (recall the definition
in~\eqref{e.Lambdapm}) is the homoclinic class of a saddle of index $u+1$.

\begin{theor}\label{th.homoclinicbarQ}
    Given the periodic point $ q^\ast= q^\ast_{I_0}\in[0,1]$ and the
expanding sequence $\xi=\xi(I_0)$ provided by
Lemma~\ref{l.newfixedexpandingpoint} applied to the interval $I_0=[f_0^{-1}(b_0),b_0]$, consider  the point
$\widehat q=\varpi^{-1}(\xi)$. Then the  periodic point $
Q^\ast=(\widehat q, q^\ast)$  has index $u+1$ and
$\La_F=H( Q^\ast,F)$.
\end{theor}

Note that this result implies in particular that this set is transitive and contains both saddles of index $u+1$ and of index $u$ and thus is not hyperbolic. To prove Theorem~\ref{th.homoclinicbarQ}, we will use the properties of the iterated function system in Section~\ref{s.onedim}. This translation is possible by the skew structure of $F$. In fact, the following remark follows immediately from this structure.

\begin{rema}\label{r.fixe}
    {\rm Given a periodic sequence $\xi=(\xi_0\ldots \xi_{m-1})^\ZZ\in\Sigma_2$ and a fixed point $ r\in [0,1]$
    of the map $f_{[\xi_0\ldots \xi_{m-1}]}$, there is a canonically associated saddle point $R=(r^s,r^u,r)$ where $(r^s,r^u)=\varpi^{-1}((\xi_0\ldots\xi_{m-1})^\ZZ)$.
    If $\lvert f_{[\xi_0\ldots \xi_{m-1}]}'(r)\rvert \ne 1$ then the saddle
$R$ is hyperbolic
 and
    \[\begin{split}
    \{r^s\}\times [0,1]^u\times W^u(r,f_{[\xi_0\ldots \xi_{m-1}]}) &\subset W^u(R,F),\\
    [0,1]^s\times \{r^u\} \times W^s(r,f_{[\xi_0\ldots \xi_{m-1}]}) &\subset W^s(R,F).
    \end{split}\]
Note that if $\lvert f_{[\xi]}'(r)\rvert > 1$ (respectively, $\lvert f_{[\xi]}'(r)\rvert < 1$) then the saddle $R$ has index $u+1$ (respectively, $u$).

We introduce some  notation. Given a periodic sequence
$(\xi_0\ldots\xi_{m-1})^\ZZ\in\Sigma_2$ and a fixed point
$r_{(\xi_0\ldots\xi_{m-1})^\ZZ} =
f_{[\xi_0\ldots\xi_{m-1}]}(r_{(\xi_0\ldots\xi_{m-1})^\ZZ})$, we
will consider the point
\[
    R_{(\xi_0\ldots\xi_{m-1})^\ZZ}\eqdef
    \big( \varpi^{-1}((\xi_0\ldots\xi_{m-1})^\ZZ),r_{(\xi_0\ldots\xi_{m-1})^\ZZ} \big)
\]
that is periodic under $F$. Notice that there can exist fibers that contain more
than one periodic point, that is, in general the points
$r_{(\xi_0\ldots\xi_{m-1})^\ZZ}$ and hence
$R_{(\xi_0\ldots\xi_{m-1})^\ZZ}$ are not unique.}
\end{rema}

\begin{rema}\label{r.suusintersect}{\rm
	Note that Remark~\ref{r.fixe} implies, in particular, that for every periodic point $R\in \La_F\setminus\{Q,P\}$ one has that $W^u(R,F)$ intersects transversely $[0,1]^s\times \{0^u\} \times (0,1)\subset W^s(P,F)$ and that $W^s(R,F)$ intersects transversely $\{0^s\}\times [0,1]^u\times (0,1)\subset W^u(Q,F)$.
}\end{rema}

\begin{rema}\label{r.RsuP}{\rm
	Given a periodic point $R=R_{(\xi_0\ldots\xi_{m-1})^\ZZ}=(r^s,r^u,r)$ such that $W^s(r,f_{[\xi_0\ldots\xi_{m-1}]})$ contains the forward orbit of either $0$ or $1$, then $W^u(P,F)$ intersects $W^s(R,F)$ transversely.
}\end{rema}

We have the following relation for periodic points with index $u$.

\begin{lemm}\label{l.homrelnegspec}
Consider a periodic sequence $\xi=(\xi_0\ldots\xi_{m-1})^\ZZ\in
\Si_2$ with $\xi\ne 0^\ZZ$ and an associated periodic point of the map $F$
$$
R=R_{(\xi_0\ldots\xi_{m-1})^\ZZ}=(r^s,r^u,r_{(\xi_0\ldots\xi_{m-1})^\ZZ})=(r^s,r^u,r).
$$
\begin{enumerate}
\item [1)]  If $R$ has index $u$ and if the stable manifold $W^s(r,f_{[\xi_0\ldots\xi_{m-1}]})$ contains $[0,1]$ then $R$ is homoclinically related to $P$.
\item [2)]
If $R$ has index $u+1$ and if the unstable manifold $W^u(r,f_{[\xi_0\ldots\xi_{m-1}]})$ contains a
fundamental domain of $f_0$ in  $(0,1)$ then for every saddle $\widetilde R\in \La_F$ with $\widetilde R\ne Q$ the manifolds $W^s(\widetilde R,F)$ and $W^u(R,F)$ intersect transversely.
\end{enumerate}
\end{lemm}

\begin{proof}
Note that as $\xi\ne 0^\ZZ$, we have that $r\in (0,1)$ and  thus $R\ne Q$, $P$.

Suppose that $R$ has index $u$. Remark~\ref{r.suusintersect} implies that $W^u(R,F)$ intersects transversely $W^s(P,F)$. To see that $W^u(P,F)$ intersects transversely $W^s(R,F)$ (and thus the $R$ and $P$ are homoclinically related) it suffices to note that $\{y^s\}\times [0,1]^u \times \{f_1(0)\}\subset W^u(P,F)$, for some $y^s\in (0,1)^s$, and that $[0,1]\subset W^s(r, f_{[\xi_0\ldots,\xi_{m-1}]})$. Remark~\ref{r.fixe} then implies the assertion in 1).

To prove item 2) note that our assumptions imply that for some fundamental domain $D=[d,f_0(d)]\subset (0,1)$ of $f_0$ and some point $z^s\in (0,1)^s$ we have that
$$
	\De\eqdef \{z^s\} \times [0,1]^u \times D \subset W^u(R,F).
$$
Remark~\ref{r.suusintersect} applied to $\widetilde R$ implies that $W^s(\widetilde R,F)$ accumulates at $W^s(Q,F)$ from the right and therefore (from the definition of $f_1$) $W^s(\widetilde R,F)$ accumulates to  $\{0^s\}\times [0,1]^u\times \{1\}$ from the left. In particular, for every $\de>0$ there are  $x\in (1-\de,1)\subset (f_0(d),1)$ and $z^u\in (0,1)$ such that
\begin{equation}\label{e.oldremark49}
	\Upsilon \eqdef [0,1]^s\times \{(z^u,x)\} \subset W^s(\widetilde R,F).
\end{equation}
The choice of $x$ implies that some negative iterate of $\Upsilon$ by $F$ transversely meets $\De\subset W^u(R,F)$. Thus $W^s(\widetilde R,F)$ transversely intersects $W^u(R,F)$, ending the proof
of the lemma.
\end{proof}

\begin{rema} \label{r.newrema49}{\rm
	Note that equation~\eqref{e.oldremark49} and the fact that
$x$ can be taken arbitrarily close to $1$ implies that for every
saddle $R\in \La_F$, $R\ne Q$, every fundamental domain $D$ of
$f_0$ in $(0,1)$, and  every $x^s\in [0,1]$ we have that
 $\{x^s\} \times [0,1]^u\times D \pitchfork W^s(R,F)\ne \emptyset$.
 }\end{rema}

We continue exploring the skew-product structure and the strong un-/stable
directions  of the global transformation $F$.

\begin{rema}\label{r.iterates}
    {\rm
    Consider an interval $I\subset[0,1]$, a point  $y^s\in [0,1]^s$, and the disk $\De=\{y^s\} \times [0,1]^u \times I$.
    Given a finite sequence $\xi=(\xi_0\ldots\xi_{m})$ with Notation~\ref{n.cylinders}   there is some  $\overline y^s\in
    [0,1]^s$ such that
    $$
    F^{m+1} (\De_{[\xi]})= \{\overline y^s\} \times [0,1]^u\times f_{[\xi]}
    (I).
    $$
}\end{rema}

\begin{lemm}\label{l.wubq}
	Given the periodic point $ Q^\ast=(q^s,q^u,q^\ast)=(\widehat q, q^\ast)$ in Theorem~\ref{th.homoclinicbarQ}, the unstable manifold $W^\ut ( Q^\ast,F)$  intersects transversely the $s$-disk $[0,1]^s\times \{(x^u,x)\}$ for any
$(x^u,x)\in [0,1]^u\times (0,1)$.
\end{lemm}

\begin{proof}
Consider the finite sequence $\xi=\xi(I_0)$ associated to $ q^\ast$ as provided by Lemma~\ref{l.newfixedexpandingpoint}. Recall that by
Lemma~\ref{l.newfixedexpandingpoint} the fundamental domain
$D=[f_0^{-2}(b_0),f_0^{-1}(b_0)]$ is contained in $W^u(q^\ast, f_{[\xi]})$. This implies that
$$
    \De_0\eqdef \{q^s\} \times [0,1]^u \times D
     \subset W^\ut (Q^\ast,F).
$$
Let us consider the following forward iterations  of  $\De_0$ by $F$. For $i\ge 0$ define recursively
\[
   \De_{i+1}
        \eqdef F^{i+1} (\De_0\cap \CC_{[0^{i+1}]})=
        F(\De_i\cap \CC_0)\\
        = \{q^s_{i+1}\}\times [0,1]^u \times f^{i+1}_0(D),
\]
for some point $q^s_{i+1}\in[0,1]^s$. Observe that
\[
	\bigcup_{i\ge 0}f_0^{i}(D)\supset \big[f_0^{-2}(b_0),1\big).
\]
Thus for every $x\in [f_0^{-2}(b_0),1)$ there is some point $y^s(x)\in [0,1]^s$ such that
$$
	\Ups_x
	= \{y^s(x)\} \times [0,1]^u \times \{x\} \subset W^\ut (Q^\ast,F).
$$
This implies that the lemma holds when $x\in [f_0^{-2}(b_0),1)$.

To complete the proof of the lemma, first observe that, by (iii) in {(F01}) for any point $x\in(0,f_0^{-2}(b_0))$ one has
\[
	x'=f_1^{-1}(x) \in \big[f_0^{-2}(b_0),1\big) .
\]	
Thus, we can consider the disk $\Ups_{x'}$ and the point $z^s(x')\in [0,1]^s$ given by
$$
    F\big(\Ups_{x'}\cap \CC_1\big)
    = \{z^s(x')\} \times [0,1]^u \times\{f_1(x')\}
    =\{z^s(x')\} \times [0,1]^u \times\{x\}.
$$
By construction, the $u$-disk  $F(\Ups_{x'}\cap \CC_1)$ is contained in $W^u( Q^\ast,F)$ and intersects the $s$-disk $[0,1]^s\times \{(x^u,x)\}$. This ends the proof of the lemma.
\end{proof}

\begin{rema}\label{r.ssleaf}
    {\rm
    Given $X=(x^s,x^u,x)\in \Lambda_F$, we denote by $W^\sst(X,F)$ (by $W^\uut(X,F)$) the strong stable manifold of $X$ (the strong unstable manifold of $X$) defined as the unique invariant manifold tangent to $E^s$ (to $E^u$) at $X$ and of dimension $s$ (dimension $u$).
Note that  we have
    $$
        [0,1]^s\times \{(x^u,x)\} \subset W^{\sst}(X,F), \qquad
\{0^s\}\times [0,1]^u \times  \{x\} \subset W^{\uut}(X,F).
    $$
    }
\end{rema}

Remark~\ref{r.ssleaf} and Lemma~\ref{l.wubq} immediately imply the
following.

\begin{coro}\label{c.wubq}
	For every $X=(x^s,x^u,x)\in \La_F$ with $x\in(0,1)$ we have $W^\ut(Q^\ast,F)\pitchfork W^{\sst}(X,F)\ne \emptyset$. In particular, we have $W^u( Q^\ast,F)\pitchfork W^s(R,F)$ for every saddle $R\in \La_F\setminus \{P,Q\}$.
\end{coro}

\begin{nota}\label{n.Lorenzo}
{\rm
    For a point $X\in \CC$ and a number $i\in\ZZ$ such that $F^i(X)\in\CC$ let us write
    \[
        X_i=F^i(X)=(x^s_i,x^u_i, x_i).
    \]
}\end{nota}

Given $x^s\in\RR^s$ we denote
$B^s_\delta(x^s)\eqdef\{x\in\RR^s\colon d(x,x^s)<\delta\}$. We
will also use the analogous notation $B^u_\delta(x^u)$.

The following proposition is the main step in the proof of Theorem~\ref{th.homoclinicbarQ}.

\begin{prop}\label{p.delta+deltaepsilon}
    Consider a point $X=(x^s,x^u,x)\in \La_F$ such that $x_i\in (0,1)$ for infinitely many $i\le 0$. Given $\delta>0$, the disk
    \[
    \De^s_\de(X)\eqdef B^s_\delta(x^s)\times \{(x^u,x)\}
    \]
    transversely intersects $W^\ut(Q^\ast,F)$.
	Given a point
	\[
		X(\delta)\eqdef (x^s(\de),x^u,x)
		\in\De^s_\de(X)\pitchfork W^\ut(Q^\ast,F),
	\]	
	then for every $\varepsilon>0$ the disk
    $$
    \De^\cut_\varepsilon (X(\de))
    \eqdef \{x^s(\de)\}\times B^u_\varepsilon(x^u) \times
    [x-\varepsilon,x+\varepsilon]\subset W^u( Q^\ast,F)
    $$
  	intersects $W^\st(Q^\ast,F)$ transversely.
\end{prop}

\begin{proof}[Proof of Proposition~\ref{p.delta+deltaepsilon}]
 Since $x_i\in (0,1)$ for infinitely many $i\le0$, the uniform expansion in the $s$-direction with respect to $F^{-1}$ implies that there is some iterate $i\le0$ such that $x_i\in(0,1)$ and
$$
	[0,1]^s\times \{(x^u_i,x_i)\}\subset F^i \big(\De^s_\de(X)\big).
$$
Thus,  by Lemma~\ref{l.wubq}, $F^i\big(\De^s_\de(X)\big)$ intersects $W^\ut (Q^\ast,F)$ transversely and hence $\De^s_\de(X)$ intersects $W^\ut(Q^\ast,F)$ transversely.

Note that, since $X\in \La_F$, the definition of $X(\delta)$ implies that $X(\de) \in \La_F^+$. Consider now the forward orbit of $X(\de)$ and let $\xi=(\xi_0\xi_1\ldots)\in\Sigma_2^+$ be the one-sided sequence that is determined by
\[
    F^i(X(\de))\in \CC_{\xi_i}, \quad i\ge 0.
\]
We let $\De_0\eqdef\De^\cut_\varepsilon(X(\delta))$ and (using Notation~\ref{n.cylinders}) recursively define for $i\ge 0$
\[
    \De_{i+1}
    \eqdef F(\De_i \cap \CC_{\xi_i})=F^{i+1}(\De_0\cap \CC_{[\xi_0\dots \xi_i]}).
\]
The uniform expansion in the $u$-direction implies that there is a
least iterate $i_0$ such that we cover the unstable vertical
direction, that is, such that
\begin{equation}\label{e.right}
    \De_{i_0}=\{y^s(\delta)\} \times [0,1]^u\times L
\end{equation}
for some point $y^s(\delta)\in[0,1]^s$ and some interval $L\subset(0,1)$. Clearly, this covering
property is also true for any $i\ge i_0$.

Notice that, in general, we have no information about the location of the interval $L$. Thus, in principle, we cannot  apply our preliminary results about expanding itineraries in Section~\ref{ss.expanding} and we need to consider some additional iterates of $L$. More precisely, let us now first consider some image $H$ of $L$ by the iterated function system such that we can apply these arguments to $H$. Recall that, in particular, such interval $H$ must be contained in $[f_0^{-2}(b_0),b_0]$. Let us take a $j_0$ large enough such that $f_0^{j_0}(L)$ is close enough to $1$ and that
\[
	f_{[0^{j_0}1]} (L)\subset \big(0,f_0^{-1}(a_0)\big) = \big(0,f_0^{-2}(b_0)\big).
\]
Consider now the smallest number $\ell_0\ge 0$ such that
\[
f_{[0^{j_0}10^{\ell_0}]} (L)\cap \big(f_0^{-2}(b_0), b_0\big]\ne\emptyset
\]
and consider the finite sequence $\eta\eqdef (0^{j_0}10^{\ell_0})$. Let us define
$$
H\eqdef f_{[\eta]}(L) \cap \big[f_0^{-2}(b_0),b_0\big]
$$
and consider the disk
\begin{equation}\label{e.eqeq}
    \widetilde \De
    \eqdef F^{j_0+1+\ell_0} \big( \De_{i_0} \cap \CC_{[\eta]} \big)
    = \{\widetilde y^s\} \times [0,1]^u  \times H,
\end{equation}
where $\widetilde y^s$ is some point in $[0,1]^s$. In comparison
to~\eqref{e.right}, this disk is now appropriate to apply our
arguments on expanding itineraries.

By Remark~\ref{r.newrema49}, if $H$ contains a fundamental domain
of $f_0$ then $\widetilde \De$ meets $W^s(\bQ,F)$  transversely
and, since $\widetilde\Delta$ is a positive iterate of
$\Delta^\cut_\varepsilon(X(\delta))$, we are done already in this
case.

In the general case we will see that some forward iterate  of $H$ will
contain the fundamental domain $[f_0^{-2}(b_0),f_0^{-1}(b_0)]$. To prove that, we apply our results about expanded successors in Section~\ref{ss.expanding}. By Lemma~\ref{l.successor}, there exist expanded successors $H=H_{\langle 0\rangle}$, $H_{\langle1\rangle}$, $\ldots$, $H_{\langle i(H)\rangle}$ of $H$ such that $H_{\langle i(H)\rangle}$ contains the fundamental domain  $[f_0^{-2}(b_0),f_0^{-1}(b_0)]$. Together with the expanded successor, for $j=0,\ldots, i=i(H)$, we obtain an expanded finite sequence $\xi_{\langle j \rangle} = \xi(H_{\langle j (H)\rangle})$ of length $\lvert \xi_{\langle j\rangle}\rvert$, recall equation \eqref{e.Ji}.

We now define recursively a sequence of disks as follows. Let $\widehat \Delta_0 \eqdef
\widetilde\Delta$ with $\widetilde\Delta$ defined
in~\eqref{e.eqeq} and for $j=0$, $\ldots$, $i(H)-1$ let
\[
    \widehat \De_{j+1}
    \eqdef F^{|\xi _{\langle j\rangle} |} (\widehat \De_j \cap \CC_{[\xi_{\langle j\rangle}]}).
\]
Notice that
\[
    \widehat \De_{j+1}
    = \{y_{j+1}^s\}\times [0,1]^u \times f_{[\xi_{\langle j\rangle}]} (H_{\langle j\rangle})
     = \{ y_{j+1}^s \}\times [0,1]^u \times H_{\langle j+1\rangle},
\]
for some point $y_{j+1}^s\in[0,1]^s$. As $H_{\langle i(H)\rangle}$
contains the fundamental domain $[f_0^{-2}(b_0),f_0^{-1}(b_0)]$ of
$f_0$, by Remark~\ref{r.newrema49} the disk $\widehat \Delta_{
i(H)}$ meets $W^s( Q^\ast,F)$ transversely. Hence, the
proposition is proved in that case also.
\end{proof}

As a consequence of the proof of
Proposition~\ref{p.delta+deltaepsilon} we obtain the following.

\begin{rema}\label{c.01}
{\rm
    Observe that we have $\{(0^s,0^u)\}\times (0,1)\subset H( Q^\ast,F)$. As the homoclinic class is a closed set, we can conclude that $\{(0^s,0^u)\}\times [0,1]\subset H( Q^\ast,F)$.
    In particular, we have $P$, $Q\in H( Q^\ast,F)$.
}\end{rema}

\begin{rema}\label{r.remax}{\rm
The proof of the proposition implies that for any hyperbolic
periodic point $R\ne Q$ of index $u+1$ the manifolds
$W^u(R,F)$ and $W^s( Q^\ast,F)$ intersect transversely.}
\end{rema}

This remark, Corollary~\ref{c.wubq}, and the fact that homoclinic relation is an
equivalence relation, together imply the following result.

\begin{coro}\label{c.homoclinicallyrelated}
	Every pair of saddles of index $u+1$ in $\La_F$ that are different from $Q$ are homoclinicaly related.
\end{coro}

We finally formulate a simple fact.

\begin{lemm}\label{l.almfin}
    Given any sequence $\xi\in\Sigma_2$, the point $(x^s,x^u)=\varpi^{-1}(\xi)$, and some point $x\in I_{[\xi]}$, we have $X=(x^s,x^u,x)\in\Lambda_F$.
\end{lemm}

\begin{proof}
    Recall Notation~\ref{n.Lorenzo}. As an immediate consequence of the skew-product structure of $F$, by definition of $I_{[\xi]}$ we have $x_i\in[0,1]$ for all $i\in\ZZ$. Since $F$ is topologically conjugate to the shift map one has that $(x^s_i,x^u_i)\in \widehat\CC$ for all $i\in\ZZ$. Hence $F^i(X)\in\CC$ for all $i\in\ZZ$ and thus $X\in\Lambda_F$.
\end{proof}

We are now ready to prove Theorem~\ref{th.homoclinicbarQ}.

\begin{proof}[Proof of Theorem~\ref{th.homoclinicbarQ}]
Clearly, given any $X=(x^s,x^u,x)\in H(Q^\ast,F)$, then $(x^s,x^u)\in\widehat\CC$.
... It hence remains to prove $\Lambda_F\subset H(Q^\ast,F)$. We consider two cases.\\[0.2cm]
\textbf{Case 1: $X=(x^s,x^u,x)\in \La_F$ and  $x_{-i}\in(0,1)$ for
infinitely many $i>0$.}
\\[0.1cm]
By Proposition~\ref{p.delta+deltaepsilon}, and using the notation there, there exists a point
\[
	X(\de)=\big(x^s(\de),x^u,x\big)
		\in\De^s_\de(X) \cap W^\ut(\bQ,F).
\]	
Note that this point belongs to the forward invariant set $\Lambda_F^+$ and that the disk $\De^\cut_\varepsilon (X(\de))\subset W^u( Q^\ast,F)$ transversely intersects $W^s( Q^\ast,F)$ and hence contains a transverse homoclinic
point of $\bQ$. Thus
\[
    \De^\cut_\varepsilon\big(X(\de)\big) \cap H(\bQ,F)\ne \emptyset.
\]
Since Proposition~\ref{p.delta+deltaepsilon} holds for any $\varepsilon>0$ we have that $X(\de)\in H( Q^\ast,F)$.
As
$\delta$ can be taken arbitrarily small, the point $X(\de)$ can be taken arbitrarily close to $X$ and thus $X\in H( Q^\ast,F)$. This implies the theorem in case 1.
\\[0.2cm]
\textbf{Case 2: There is $i_0$ such that $X=(x^s,x^u,x)\in \Lambda_F$ and $x_{-i}\in\{0,1\}$ for all $i\ge i_0$.}\\[0.2cm]
Replacing $X$ by its iterate $F^{-i_0}(X)$, we can assume that
$x_{-i}\in\{0,1\}$ for all $i\ge 0$.
We continue distinguishing yet two more cases.\\[0.2cm]
\textbf{Case 2.1: $x_0=1$.} \\[0.2cm]
Since $f_\ell(x)=1$ if and only if $\ell=0$ and $x=1$, as the only possibility for the backward branch of $X$ we must have $x_{-i}=1$ for all $i\ge 0$. Moreover, the sequence $\xi =\varpi^{-1}(x_0)$ must satisfy $\xi_{-i}=0$ for all $i\ge0$. Hence $X_{-i}\in\CC_0$ for all $i\ge 0$ and therefore the point $X$ is of the form $X=(0^s,x^u,1)$.

Note that  $I_{[\xi]}=[0,1]$. Thus, by Lemma~\ref{l.almfin}, given any $\tau>0$
for every  $y\in (1-\tau,1)$ there is a point $Y=(0^s,x^u,y)$ that is in the set $\Lambda_F$. Note that these points form a uncountable set. Since
the set of all pre-images
\[
    \big\{ f_{[\xi_{-m}\ldots\xi_{-1}.]}\big(\{0,1\}\big)\colon
        ( \xi_{-m}\ldots\xi_{-1})\in\{0,1\}^{-m}, m\ge 1\big\}
\]
is countable, without loss of generality we can assume that the
point $Y$ and its preimages with central coordinates $y_{-i}$
additionally satisfies $y_{-i}\in(0,1)$ for all $i\ge 0$.
 Now we apply Case 1 to the point $Y$ and can conclude that $Y\in H( Q^\ast,F)$. Since a homoclinic class is a closed set and $Y$ can be chosen arbitrarily close to $X$, we yield $X\in H( Q^\ast,F)$.\\[0.2cm]
\textbf{Case 2.2: $x_0=0$.} \\[0.2cm]
To distinguish the two only possible types of backward branches of $X$ in this case, observe that $f_\ell(x)=0$ if either $x=0$ and $\ell=0$ or $x=1$ and  $\ell=1$.\\[0.2cm]
\textbf{Case 2.2 a:} We have  $x_{-i}=0$ for all $i\ge 0$. Hence in this case $X_{-i}\in\CC_0$ for all $i\ge 0$, and we can conclude as in Case 2.1.\\[0.2cm]
\textbf{Case 2.2 a:} There exists a first index $i$ such that
$x_{-i}=1$. Replacing $X$ by the iterate $f^{-i}(X)$,
we can now conclude as in Case 2.1.\\[0.2cm]
This proves that $\Lambda_F\subset H(Q^\ast,F)$ and hence the proves the theorem.
\end{proof}

%------------------------------------------------------------------------------------------
\subsection{Particular cases}
%------------------------------------------------------------------------------------------

Supplementing the results in the previous section we show that, under additional mild hypotheses on the maps $f_0$, $f_1$, we yield further properties of the homoclinic class.

First, we assume that the following Kupka Smale-like condition is satisfied.

\begin{itemize}
\item[(F${}_{\rm KS}$)]
	Every periodic point of any composition $f_{[\xi_0\ldots\xi_m]}$ is hyperbolic.
\end{itemize}
Note that this condition is generic among the pair of maps $f_0$, $f_1$ satisfying  conditions~(F0),~(F1), and~(F01).

\begin{theoo}\label{t.KS}
	Under the additional hypothesis (F${}_{\rm KS}$), for every periodic sequence $\xi=(\xi_0\ldots\xi_{m-1})^\ZZ$ there is a periodic point $R_\xi$ of $F$ of index $u$
	 in the fiber of $\xi$
(that is, $\pi(R_\xi)=\xi$) that is homoclinically related to $P$.
\end{theoo}

\begin{proof}
The arguments in the proof of Lemma~\ref{l.cont} and the hypothesis (F${}_{\rm KS}$) together imply that the set
$$
	\bigcap_{k\in \NN} (f_{[\xi_0\ldots \xi_{m}]})^{2k} ([0,1])
$$
is either an attracting fixed point of $f_{[\xi_0\ldots \xi_{m}]}$ or an interval whose
extremes are hyperbolic attracting periodic points of $f_{[\xi_0\ldots
\xi_{m}]}$. In either of these two cases let us consider one such attracting point and denote it by $r_\xi$. By construction $W^s(r,f_{[\xi_0\ldots\xi_{m}]})$ either contains $0$ or $1$. By
Remark~\ref{r.RsuP}, this implies that $W^s(R_\xi,F)$ transversely
intersects $W^u(P,F)$.
On the other hand, by Remark~\ref{r.suusintersect} we know that $W^u(R_\xi,F)$ and $W^s(P,F)$ intersect transversely. This implies that the saddles $R_\xi$ and $P$ are homoclinically related proving the proposition.
\end{proof}

Recall that in the previous case under the conditions (F0), (F1), (F01) we have $H(P,F)\subset\Lambda_F$. We now consider another particular case. Let us assume that the maps $f_0$ and $f_1$ satisfy the following condition.

\begin{itemize}
	\item[(F${}_{\rm B}$)] If $f_1([0,1])=[0,c\,]$ then $f_0'(x)\in (0,1)$ for all $x\in[c,1]$.
\end{itemize}

\begin{theoo}\label{t.blenders}
	Under the additional hypothesis (F${}_{\rm B}$) we have $H(P,F)=\Lambda_F= H( Q^\ast,F)$.
\end{theoo}

The proof follows a line of argument somewhat analogous to the one of Theorem~\ref{th.homoclinicbarQ}.  Moreover, our arguments follow very closely the exposition in~\cite[Section 6.2]{BonDiaVia:05} using a construction of so-called \emph{blenders}.

First, we have a completely analogous version of Proposition~\ref{p.delta+deltaepsilon}.

\begin{prop}\label{p.delta+deltaepsilonbis}
	Consider a point $Y=(y^s,y^u,y)\in \La_F$ such that $y_i\in (0,1)$ for infinitely many $i\ge 0$.  Given $\delta>0$, the disk
	$$
		\De^u_\de (Y)\eqdef \{y^s\} \times B^u_\de(y^u) \times \{y\}.
	$$
	transversely intersects $W^s(P,F)$. Given a point
	\[
		Y(\de) \eqdef \big(y^s, y^u(\de),y\big)\in
		\De^u_\de (Y) \pitchfork W^s(P,F),
	\]	
	then for every small $\varepsilon>0$ the disk
	$$
		\De^{cs}_\varepsilon (Y(\de))\eqdef
		B^s_\varepsilon(y^s) \times \{y^u(\de)\} \times [y-\varepsilon,y+\varepsilon]
		\subset W^s(P,F)
	$$
	intersects $W^u(P,F)$ transversely.
\end{prop}

After proving this proposition the proof of Theorem~\ref{t.blenders} is identical to the one of Theorem~\ref{th.homoclinicbarQ}, so we refrain from giving these details.

\begin{proof}[Proof of Proposition~\ref{p.delta+deltaepsilonbis}]
	The first steps of the proof are identical to the ones of the proof of Proposition~\ref{p.delta+deltaepsilon}.

Further, to show that we have $ \De^{cs}_\varepsilon(Y(\de))\pitchfork W^s(P,F)\ne\emptyset$, we consider the iterate of an interval $J\eqdef [y-\varepsilon, y+\varepsilon]$ by maps $f_{[\xi_{-m}\ldots\xi_{-1}.]}$.
First note that, under the hypothesis (F${}_{\rm B}$), the maps $f_0^{-1}$ and $f_1^{-1}$ are uniformly expanding maps in $[c,1]$ and $[0,c]$, respectively, with derivatives having moduli $\ge\kappa>1$. It is hence an immediate consequence (see also the Lemma in~\cite{BonDiaVia:95})  that $f_{[\xi_{-m}\ldots\xi_{-1}.]}(J)$ for large $m$ eventually contains the point $c$.
Just observe that if $c\notin f_{[\xi_{-m}\ldots\xi_{-1}.]}(J)$ then either $f_{[\xi_{-m}\ldots\xi_{-1}.]}(J)\in[0,c)$ or $f_{[\xi_{-m}\ldots\xi_{-1}.]}(J)\in(c,1]$. In the first case let $\xi_{-m-1}=0$ while in the second one let $\xi_{-m-1}=1$.  Consequently, $\lvert  f_{[\xi_{-m-1}\ldots\xi_{-1}.]}(J)\rvert\ge \kappa^{m+1} \lvert J\rvert$. Thus, there is a first $m$ with the desired property.

Now, to finish the proof, note that the skew-product structure hence implies that there exist points ${\widetilde y\,}^u\in [0,1]^u$ such that
$$
	[0,1]^s\times \{{\widetilde y\,}^u\} \times f_{[\xi_{-m}\ldots\xi_{-1}.]}(J)
	\subset W^s(P,F)
$$
and, recalling that $c=f_1(0)$, it implies that there is ${\widetilde y\,}^s\in [0,1]^s$ so that
$$
	\{{\widetilde y\,}^s\} \times [0,1]^u \times \{f_1(0)\}
	\subset W^u(P,F).
$$
This means that we have $\De^{cs}_\varepsilon (Y(\de)) \pitchfork W^u(P,F)\ne\emptyset$.
\end{proof}

%------------------------------------------------------------------------------------------
\section{Lyapunov exponents in the central direction}\label{sec:lyap}
%------------------------------------------------------------------------------------------

We now continue our discussion of Lyapunov exponents started in
Section~\ref{ss.lyap}. Recall that, due to the skew product
structure and our hypotheses, the splitting in~\eqref{e.splitt} is
dominated and for every Lyapunov regular point coincides with the
Oseledec splitting provided by the multiplicative ergodic
theorem. Here, in particular, a point $S$ is \emph{Lyapunov regular}
if and only if for $i=uu$, $c$, and $ss$ for every $v\in E^i_S$
the limit
\begin{equation}\label{e.lyapli}
    \chi_i(S)\eqdef
    \lim_{n\to\pm\infty}\frac{1}{n}\log\,\lVert dF^n_S(v)\rVert
\end{equation}
exists. In the following we will focus only on the
\emph{Lyapunov exponent $\chi_{\rm c}(S)$ associated to the central direction
$E^c$}. Observe that given a Lyapunov regular
point $S=(s^s,s^u,s)\in \La_F$ and
$\xi=(\ldots\xi_{-1}.\xi_0\xi_1\ldots)\in\Sigma_2$ given by
$\xi=\varpi((s^s,s^u))$, we have
\begin{equation}\label{e.lyapkat}
    \chi_c(S)
    = \lim_{n\to\infty}\frac 1 n \log\,\lvert (f_{[\xi_0\ldots\, \xi_{n-1}]})'(s)\rvert.
\end{equation}
Clearly, $\chi_c(S)$ is well-defined for every periodic point.

%------------------------------------------------------------------------------------------
\subsection{Spectra of Lyapunov exponents}\label{sec:gapcentral}
%------------------------------------------------------------------------------------------

Let us consider spectra of central exponents from various points of view.

%------------------------------------------------------------------------------------------
\subsubsection{Spectrum related to periodic points}
%------------------------------------------------------------------------------------------
Given a saddle $S$, we define the \emph{spectrum of  saddles homoclinically related to $S$} by
\[
    \cL_{\rm homrel}(S) \eqdef
    \big\{\chi_c(R)\colon R\text{ hyperbolic periodic homoclinically related to }S\big\}
\]
and the \emph{periodic point spectrum of the homoclinic class of $S$} by
\[
    \cL_{\rm per}\big(H(S,F)\big) \eqdef
    \big\{\chi_c(R)\colon R\in H(S,F)\text{ periodic}\big\}.
\]
Clearly, $\cL_{\rm homrel}(S) \subset\cL_{\rm per} (H(S,F))$. Let ${Q^\ast}$ be the saddle provided by Theorem~\ref{th.homoclinicbarQ}. Since the homoclinic
class $H( Q^\ast,F)$ coincides with the maximal invariant set
$\Lambda_F$, we have
\begin{equation}\label{e.spec}
    \cL_{\rm homrel}(Q) \cup \cL_{\rm homrel}( Q^\ast)
    \cup \cL_{\rm homrel}(P)
    \subset \cL_{\rm per}\big(H( Q^\ast,F)\big).
\end{equation}
Moreover, $\cL_{\rm homrel}(Q)=\{\log\beta\}$ by Lemma~\ref{l.hq}.

Let us recall the following standard fact (see also~\cite[Corollary 2]{AbdBonCroDiaWen:07}).

\begin{lemm}\label{l.homrelspec}
    Given two saddles  $S$ and $S'$ that are homoclinically related and satisfy $\chi_c(S)\le\chi_c(S')$, we have
    \[
        [\chi_c(S),\chi_c(S')]
        \subset \overline{\cL_{\rm homrel}(S)}
        = \overline {\cL_{\rm homrel}(S')}.
    \]
\end{lemm}

\begin{proof}
    Recall that if the saddles $S$ and $S'$ are  homoclinically related then
    there exists a horseshoe $\Lambda_{S,S'}\subset \Lambda_F$
    that contains both saddles. In particular $\Lambda_{S,S'}$
    is a uniformly hyperbolic locally maximal and transitive set with respect to $F$.
    The existence of a Markov partition implies that we can construct orbits in
    the hyperbolic set that spend a fixed proportion of time close to $S$ and $S'$, respectively.
    This is enough to obtain periodic points in $\Lambda_{S,S'}$ with Lyapunov exponents dense  in the interval $[\chi_c(S),\chi_c(S')]$.
    Finally,   any such periodic orbit is homoclinically related to $S$ and to $S'$.
 \end{proof}

%------------------------------------------------------------------------------------------
\subsubsection{Spectrum of Lyapunov regular points}
%------------------------------------------------------------------------------------------
Let us define
\[
	\cL_{\rm reg}\big(H(S,F)\big)
	\eqdef\big\{\chi_{\rm c}(R)\colon R\in H(S,F) \text{ Lyapunov regular}\big\}.
\]
We finally obtain the possible spectrum of central Lyapunov exponents.
Recall the definition of $\widetilde\beta$ in Proposition~\ref{p.beta} that is the biggest Lyapunov exponent as in~\eqref{e.lyapkat} that is different from $\beta$.

\begin{prop}\label{p.spectrum}
	Let
	\[
		\beta^\ast\eqdef
		\exp\sup\big\{\chi\colon\chi
		\in \cL_{\rm reg}\big(H( Q^\ast, F)\big), \chi\ne\log\beta\big\}.
	\]
	We have
	\[\begin{split}
		\overline{\cL_{\rm per}\big(H( Q^\ast,F)\big)}
		&= \big[\log\lambda,\log\beta^\ast\big]\cup\{\log\beta\}\\
		&\subset \overline{\cL_{\rm reg}\big(H(Q^\ast,F)\big)}
		\subset \big[\log\lambda,\log\widetilde\beta\,\big]\cup\{\log\beta\}.
 	\end{split}\]
\end{prop}

\begin{rema}{\rm
	Using different methods involving shadowing-like properties, one can in fact show that $[\log\lambda,\log\widetilde\beta\,]\cup\{\log\beta\}$ is set of all possible upper/lower central Lyapunov exponents, that is, all exponents that are obtained when we replace $\lim$ by $\limsup$/$\liminf$ in~\eqref{e.lyapli}. Hence, in particular, we have equalities in Proposition~\ref{p.spectrum}. For the details refer to~\cite{DiaGelRam:}.
}\end{rema}

\begin{proof}[Proof of Proposition~\ref{p.spectrum}]
Let us first prove that $(0,\log\beta^\ast)\subset\cL_{\rm per}(H( Q^\ast,F))$. Note that by Proposition~\ref{p.posspec} for every $\varepsilon>0$ there exist a finite sequence $(\xi_0\ldots\xi_{m-1})$ and a fixed point $q_{(\xi_0\ldots\xi_{m-1})^\ZZ}\in(0,1)$ with respect to the map $f_{[\xi_0\ldots\xi_{m-1}]}$ that has Lyapunov exponent in
$(0,\varepsilon)$. Therefore, the corresponding periodic point $Q_{(\xi_0\ldots\xi_{m-1})^\ZZ}$ has central Lyapunov exponent in $(0,\varepsilon)$. By Corollary~\ref{c.homoclinicallyrelated} this point is homoclinically related to $ Q^\ast$. By
Lemma~\ref{l.homrelspec} and~\eqref{e.spec} we hence obtain that
\[
	[0,\chi_c(Q^\ast)]\subset \overline{\cL_{\rm per}\big(H( Q^\ast,F)\big)}.
\]	
Similarly, we obtain that
\[
	[\chi_c( Q^\ast),\log\beta^\ast]
	\subset \overline{\cL_{\rm per}\big(H( Q^\ast,F)\big)}
\]	
proving that
\[
	[0,\log\beta^\ast] \subset\overline{\cL_{\rm per}\big(H( Q^\ast,F)\big)}.
\]	

Now we consider the negative part of the spectrum. Note that by Proposition~\ref{p.lajk} for every $\varepsilon>0$ there exist a finite sequence $(\xi_0\ldots\xi_{m-1})$ and a fixed point $p_{(\xi_0\ldots\xi_{m-1})^\ZZ}$ with respect to the map $f_{[\xi_0\ldots\xi_{m-1}]}$ that has Lyapunov exponent in $(-\varepsilon,0)$ and whose stable manifold contains the interval $[0,1]$. By item 1) in Lemma~\ref{l.homrelnegspec} the corresponding hyperbolic periodic point $P_{(\xi_0\ldots\xi_{m-1})^\ZZ}$ is homoclinically related to the fixed point $P$. Exactly as above, we obtain
\[
	[\log\lambda,0]\subset\overline{\cL_{\rm per}\big(H(P,F)\big)}.
\]	
Since $H(P,F)\subset H(Q^\ast,F)$, this proves
\[
	[\log\lambda,\log\beta^\ast]\cup\{\log\beta\}
	\subset \overline{\cL_{\rm per}\big(H( Q^\ast,F)\big)}.
\]	
By definition of $\beta^\ast$ we have
\[
	\cL_{\rm per}\big(H( Q^\ast,F)\big)
	\subset [\log\lambda,\log\beta^\ast\,]\cup\{\log\beta\}.
\]	

Clearly, $\cL_{\rm per}(H( Q^\ast,F))\subset \cL_{\rm reg}(H( Q^\ast,F))$.
Finally, note that by Proposition~\ref{p.beta}, any Lyapunov regular point is either contained in the stable manifold of $Q$ and hence has exponent $\log\beta$ or has exponent less or equal than $\log\widetilde\beta$, proving
\[
	\cL_{\rm reg}\big(H( Q^\ast,F)\big)
	\subset [\log\lambda,\log\widetilde\beta\,]\cup\{\log\beta\}.
\]
This finishes the proof of the proposition.	
\end{proof}

%------------------------------------------------------------------------------------------
\subsubsection{Spectrum of ergodic measures}
%------------------------------------------------------------------------------------------
The following results we will need in the following section. We denote by $\cM(\Lambda)$ the set of $F$-invariant Borel probability measures supported on a set $\Lambda$ and by $\cM_{\rm erg}(\Lambda)$ the subset of ergodic measures. For $\mu\in\cM(\Lambda)$ let
\[
    \chi_c(\mu)\eqdef \int\log\,\lVert dF|_{E^c}\rVert\,d\mu.
\]
Denote by $\de_Q$ the Dirac measure at $Q$ and consider
\[
    \cL_{\rm erg}\big(H( Q^\ast,F)\big)
    \eqdef \big\{\chi_c(\mu)\colon
        \mu\in\cM_{\rm erg}(H( Q^\ast,F))\setminus\{\delta_Q\}\big\}.
\]
The following is an immediate consequence of Proposition~\ref{p.beta}.

\begin{prop}\label{l.extremelya}
    We have
    $\displaystyle
        \big(\log\widetilde\beta
        ,\log\beta\big)\cap
        \cL_{\rm erg}(H( Q^\ast,F))
        =\emptyset.
    $
\end{prop}

%------------------------------------------------------------------------------------------
\subsection{Phase transitions}\label{sec:phase}
%------------------------------------------------------------------------------------------

In this section we continue our analysis of spectral properties
and study equilibrium states. Recall that, given a continuous
potential $\varphi\colon\Lambda_F\to\RR$, a $F$-invariant Borel
probability measure $\nu$ is called an \emph{equilibrium state} of
$\varphi$ with respect to $F|_{\Lambda_F}$ if
\[
    h_\nu(F)+\int\varphi\,d\nu
    = \max_{\mu\in\cM(F|\Lambda_F)}\Big( h_\mu(F)+\int\varphi\,d\mu\Big),
\]
where $h_\mu(F)$ denotes the measure theoretic entropy of the measure $\mu$. Notice that as the central direction has dimension one such maximizing measure indeed exists by~\cite[Corollary 1.5]{DiaFis:} (see also~\cite{CowYou:05}).
Without loss of generality, we can always assume that such measure is ergodic. Indeed, given an equilibrium state for $\varphi$ that is non-ergodic, any ergodic measure in its ergodic decomposition is also an equilibrium state for $\varphi$.
Note that
\begin{equation}\label{varprinc}
    P(\varphi) =
    \max_{\mu\in\cM(F|\Lambda_F)}\Big( h_\mu(F)+\int\varphi\,d\mu\Big)
\end{equation}
is the \emph{topological pressure} of $\varphi$ with respect to $F|_{\Lambda_F}$
(see~\cite{Wal:81}). An equilibrium state for the zero potential $\varphi=0$ is simply a \emph{measure of maximal entropy} $h(F)=h(F|_{\Lambda_F})$.

We will study the following family of continuous potentials
\[
    \varphi_t\eqdef -t\log\,\lVert dF|_{E^c}\rVert,
    \quad t\in\RR,
\]
and will continue denoting $P(t)=P(\varphi_t)$. Note that $t\mapsto P(t)$ is convex (and hence continuous and differentiable on a residual set, see also~\cite{Wal:81} for further details). Recall that a number $\alpha\in \RR$ is said to
be a \emph{subgradient}  at $t$ if we have $P(t+s)\ge P(t)+s\, \alpha$ for all $s\in \RR$.

\begin{lemm}\label{l.P2}
    For any number $t\in\RR$ and any equilibrium state $\mu_t$ of the potential $\varphi_t$ the number $-\chi_c(\mu_t)$ is a subgradient of $s\mapsto P(s)$ at $s=t$.  If moreover $s\mapsto P(s)$ is differentiable at $s=t$ then $\chi_c(\mu_t)=-P'(t)$.
\end{lemm}

\begin{proof}
Given $t\in \RR$ and an equilibrium state $\mu_t$, it follows from
the variational principle~\eqref{varprinc} that for all $s\in\RR$
we have
    \[
    P(t+s)\ge
    h_{\mu_t}(F)-(t+s)\,\chi_c(\mu_t) =
    P(t) - s\,\chi_c(\mu_t),
    \]
that is, $-\chi_c(\mu_t)$ is a subgradient at $t$.

If the pressure is differentiable at $t$ then this subgradient is unique and thus all equilibrium states of $\varphi_t$ have the same exponent given by $-P'(t)$.
\end{proof}

We now derive  the existence of a \emph{first order phase transition}, that is, a parameter $t$ at which the pressure function is not differentiable.
Note that in our case, by Lemma~\ref{l.P2}, this is equivalent to the existence of a parameter $t$ and (at least) two equilibrium states for $\varphi_t$ with different central exponent.

\begin{prop}\label{p.phasetrans}
    There exists a parameter $t_Q\in\big[-h(F)/(\log\beta-\log\widetilde\beta),0\big)$
 so that  for every $t\le t_Q$ the measure $\delta_Q$ is an equilibrium state for $\varphi_t$ and $P(t)=-t\log\beta$. Moreover, $s\mapsto P(s)$ is not differentiable at $s=t_Q$ and
    \[
        D^-P(t_Q)=-\log\beta\quad\text{ and }\quad
        D^+P(t_Q)\ge-\log\widetilde\beta.
    \]
\end{prop}

\begin{proof}
Let us first show that $\delta_Q$ is an equilibrium state for some parameter $t$.
Recall that (F0.i) implies that $\chi_c(\delta_Q)=\log\beta$.
Note that the variational principle~\eqref{varprinc} implies that for every $t\in\RR$ we have
    \begin{equation}\label{e.contrad}
        P(t)\ge h_{\delta_Q}(F)-t\chi_c(\delta_Q)=-t\log\beta.
    \end{equation}
Arguing by contradiction, assume that there exists no $t\in\RR$ such that $\delta_Q$ is an equilibrium state for $\varphi_t$. Then, by~\eqref{varprinc}, $P(t)>-t\log\beta$ for all $t$. By Proposition~\ref{l.extremelya} there is no ergodic $F$-invariant measure different from $\delta_Q$ with central Lyapunov exponent within the interval $(\log\widetilde\beta,\infty)$.
Hence, for any $t\in\RR$ any ergodic equilibrium state of $\varphi_t$ different from $\de_Q$  has exponent $\le\log\widetilde\beta$. In particular, for every $t<0$ we have
\begin{equation}\label{e.samesame}
    P(t)
    = h_{\mu_t}(F)+\lvert t\rvert\,\chi_c(\mu_t)
    \le h(F) +\lvert t\rvert\log\widetilde\beta,
\end{equation}
for some ergodic equilibrium state $\mu_t$ of $\varphi_t$.
Summarizing, for $t<0$ we have
\begin{equation}\label{e.saame}
    \lvert t\rvert\,\log\beta<P(t)\le h(F)+\lvert t\rvert\,\log\widetilde\beta.
\end{equation}
But this is a contradiction if $t<0$ and $\lvert t\rvert< h(F)/(\log\beta-\log\widetilde\beta)$. Therefore, there exist $t\in\RR$ so that $\delta_Q$ is an equilibrium state for $\varphi_t$.

Since $P(0)=h(F)\ge\log 2>0$, by continuity of $t\mapsto P(t)$ we have
\[
	t_Q\eqdef \max\{t\in\RR\colon P(t)=-t\log\beta\}<0.
\]	
Consider $\tau\in(t_Q,0)$ and an ergodic equilibrium state $\mu_\tau$ for $\varphi_\tau$.
Since $P(\tau)\ne -\tau\log\beta$ we have $\mu_\tau\ne\delta_Q$. Hence, by Proposition~\ref{l.extremelya}, we have $\chi(\mu_\tau)\le\log\widetilde\beta$. Note again that $P(\tau)> -\tau\log\beta$. Thus, arguing exactly as in~\eqref{e.samesame} and~\eqref{e.saame}, we yield $\lvert\tau\rvert< h(F)/(\log\beta-\log\widetilde\beta)$. In particular, this implies
\[
	\lvert t_Q\rvert \le \frac{h(F)}{\log\beta-\log\widetilde\beta}.
\]	
This completes the first part of the lemma.

What remains to show are the estimates for the left and right derivatives at $t_Q$.
Consider again a number $\tau\in(t_Q,0)$ and an ergodic equilibrium state $\mu_\tau$ for $\varphi_\tau$. As above, $\mu_\tau\ne\delta_Q$ and $\chi_{\rm c}(\mu_\tau)\le\log\widetilde\beta$. By Lemma~\ref{l.P2}, $-\chi_{\rm c}(\mu_\tau)$ is a subgradient at $\tau$. Hence, $D^+P(t_Q)\ge-\log\widetilde\beta$. On the other hand, by definition of $t_Q$, we have $D^-P(t_Q)=-\log\beta$. Hence $t\mapsto P(t)$ is not differentiable at $t_Q$.
\end{proof}

%-------------------------------------------------------------------------------------------------------
\section{Proofs of the main results}\label{sec:ptheo}
%-------------------------------------------------------------------------------------------------------

\begin{proof}[Proof of Theorem~\ref{theorem0}]
	Items (A) and (C) follow immediately from Propositions~\ref{p.residualtrivfib} and~\ref{p.beta}, respectively.
	Finally, item (B) is an one-dimensional version of item (B) in Theorem~\ref{t.bonyexample} using also transitivity in item (D), so we will omit its proof.
\end{proof}

\begin{proof}[Proof of Theorem~\ref{t.bonyexample}]
The item (A.a) follows from item (A) in Theorem~\ref{theorem0} together with the skew-product structure of $F$.

By Theorem~\ref{th.homoclinicbarQ} there is a saddle $ Q^\ast$ of index $u+1$ such that the homoclinic class $H=H( Q^\ast,F)$ coincides with the locally maximal invariant set $\Lambda_F$ in $\CC$. This homoclinic class coincides with the closure of all saddle points homoclinically related to $Q^\ast$. By Corollary~\ref{c.homoclinicallyrelated} every saddle of index $u+1$ in $\Lambda_F$ is homoclinically related to $Q^\ast$. By Lemma~\ref{l.cont} the spine of every periodic point of index $u+1$ is non-trivial. Hence, the set of all points with non-trivial spines is dense in $\Lambda_F$. This shows item (A.b).

The first part of item (B) follows from the previous arguments. Recall that $P$ has index $u$ and that its homoclinic class of $P$ is non-trivial (Lemma~\ref{l.hp}) and therefore contains infinitely many saddles of index $u$. Since this class is  contained in $\Lambda_F$ we are done.

The first part of item (C) follows from item (C) in Theorem~\ref{theorem0} together with the skew-product structure of $F$. The fact that the spectrum contains an interval and the existence of a phase transition follow immediately from Propositions~\ref{p.spectrum} and~\ref{p.phasetrans}.

Again, by Theorem~\ref{th.homoclinicbarQ} for the saddle $ Q^\ast$ of index $u+1$ the homoclinic class $H=H( Q^\ast,F)$ is the locally maximal set in $\CC$ and hence contains the non-trivial class $H(P,F)$. Finally, Lemmas~\ref{l.hq} implies that $H(Q,F)=\{Q\}\subset H$. This proves item (D) and hence the theorem.
\end{proof}

%-------------------------------------------------------------------------------------------------------

\bibliographystyle{amsplain}

\end{document}